\pdfoutput=1
\documentclass[11pt]{article}
\usepackage{authblk}

\usepackage{color}
\usepackage{helvet}         
\usepackage{courier}        
\usepackage{graphicx}        
\usepackage{amsmath}
\usepackage{amssymb}
\usepackage{bbold}
\usepackage{amsthm}
\usepackage{subcaption}
\usepackage{sidecap}
\usepackage{comment}
\usepackage[font=small]{caption}

\usepackage{yhmath}
\usepackage{thmtools}
\usepackage{thm-restate}
\usepackage{floatrow}
\usepackage{setspace}

\newtheorem{theorem}{Theorem}
\newtheorem{corollary}[theorem]{Corollary}
\newtheorem{lemma}[theorem]{Lemma}
\newtheorem{conjecture}[theorem]{Conjecture}
\newtheorem{proposition}[theorem]{Proposition}

\newtheorem{remark}[theorem]{Remark}
\newtheorem{definition}[theorem]{Definition}
\usepackage[top=2cm, bottom=2cm, left=2cm, right=2cm]{geometry}

\numberwithin{theorem}{section}
\numberwithin{conjecture}{section}
\numberwithin{corollary}{section}
\numberwithin{lemma}{section}
\numberwithin{proposition}{section}
\numberwithin{remark}{section}
\numberwithin{definition}{section}

\numberwithin{figure}{section}
\numberwithin{equation}{section}

\begin{document}

\title{On the Uniqueness of Global Multiple SLEs}


\vspace{5cm}

\begin{center}
\LARGE \bf On the Uniqueness of Global Multiple SLEs
\end{center}

\vspace{0.25cm}

\begin{center}
{\bf Vincent Beffara}\\
{\footnotesize{\texttt{vincent.beffara@univ-grenoble-alpes.fr}}}\\
{\small{Universit\'{e} de Grenoble Alpes, CNRS, Institut Fourier, Grenoble, France}}

\bigskip

{\bf Eveliina Peltola}\\
{\footnotesize{\texttt{eveliina.peltola@hcm.uni-bonn.de}}}\\
{\small{Institute for Applied Mathematics, University of Bonn, Germany}}

\bigskip

{\bf Hao Wu} \\
{\footnotesize{\texttt{hao.wu.proba@gmail.com}}}\\
{\small{Yau Mathematical Sciences Center, Tsinghua University, China}}
\end{center}

\vspace{0.75cm}

\begin{center}
\begin{minipage}{0.95\textwidth}
\abstract{
    This  article  focuses on  the  characterization  of
    global multiple  Schramm-Loewner evolutions  (SLE). The  chordal SLE
    describes the scaling limit of a single interface in various
    critical  lattice models  with  Dobrushin  boundary conditions,  and
    similarly,  global   multiple  SLEs   describe  scaling   limits  of
    collections   of  interfaces   in  critical   lattice  models   with
    alternating boundary conditions. In this  article, we give a minimal
    amount of characterizing properties for the global multiple SLEs: we
    prove that there exists a  unique probability measure on collections
    of  pairwise  disjoint  continuous  simple  curves  with  a  certain
    conditional  law   property.  As   a  consequence,  we   obtain  the
    convergence  of  multiple  interfaces  in  the  critical  Ising,
    FK-Ising, and percolation models.
}
\end{minipage}
\end{center}


\tableofcontents

\newpage

\global\long\def\ud{\mathrm{d}}
\global\long\def\der#1{\frac{\ud}{\ud#1}}
\global\long\def\pder#1{\frac{\partial}{\partial#1}}
\global\long\def\pdder#1{\frac{\partial^{2}}{\partial#1^{2}}}

\global\long\def\PartF{\mathcal{Z}}
\global\long\def\CobloF{\mathcal{U}}
\global\long\def\chamber{\mathfrak{X}}

\global\long\def\Catalan{\mathrm{C}}
\global\long\def\LP{\mathrm{LP}}
\global\long\def\DP{\mathrm{DP}}
\global\long\def\DPleq{\preceq} 
\global\long\def\DPgeq{\succeq} 
\newcommand{\wedgeat}[1]{\lozenge_#1} 
\newcommand{\upwedgeat}[1]{\wedge^#1}
\newcommand{\downwedgeat}[1]{\vee_#1}
\newcommand{\slopeat}[1]{\times_#1}
\newcommand{\removewedge}[1]{\setminus \wedgeat{#1}}
\newcommand{\removeupwedge}[1]{\setminus \upwedgeat{#1}}
\newcommand{\removedownwedge}[1]{\setminus \downwedgeat{#1}}
\newcommand{\wedgelift}[1]{\uparrow \wedgeat{#1}} 
\global\long\def\Mmat{\mathcal{M}}
\global\long\def\Minv{\mathcal{M}^{-1}}
\global\long\def\link#1#2{\{#1,#2\}}
\global\long\def\removeLink{/}
\global\long\def\nested{\boldsymbol{\underline{\Cap}}}
\global\long\def\unnested{\boldsymbol{\underline{\cap\cap}}}

\global\long\def\Rpos{\R_{> 0}}
\global\long\def\Znn{\Z_{\geq 0}}
\global\long\def\im#1{\operatorname{Im}(#1)}

\global\long\def\localSLE{\mathsf{P}}

\global\long\def\FKdual{\mathcal{L}}

\global\long\def\graph{\mathcal{G}}

\newcommand{\eps}{\varepsilon}
\newcommand{\ov}{\overline}
\newcommand{\U}{\mathbb{U}}
\newcommand{\T}{\mathbb{T}}
\newcommand{\HH}{\mathbb{H}}
\newcommand{\LA}{\mathcal{A}}
\newcommand{\LB}{\mathcal{B}}
\newcommand{\LC}{\mathcal{C}}
\newcommand{\LD}{\mathcal{D}}
\newcommand{\LF}{\mathcal{F}}
\newcommand{\LK}{\mathcal{K}}
\newcommand{\LE}{\mathcal{E}}
\newcommand{\LG}{\mathcal{G}}
\newcommand{\PL}{\mathcal{L}}
\newcommand{\LM}{\mathcal{M}}
\newcommand{\LN}{\mathcal{N}}
\newcommand{\LQ}{\mathcal{Q}}
\newcommand{\PE}{\mathcal{P}}
\newcommand{\LT}{\mathcal{T}}
\newcommand{\LS}{\mathcal{S}}
\newcommand{\LU}{\mathcal{U}}
\newcommand{\LV}{\mathcal{V}}
\newcommand{\LH}{\mathcal{H}}
\newcommand{\LX}{\mathcal{X}}
\newcommand{\Lx}{\mathbb{x}}
\newcommand{\Ly}{\mathbb{y}}
\newcommand{\R}{\mathbb{R}}
\newcommand{\C}{\mathbb{C}}
\newcommand{\N}{\mathbb{N}}
\newcommand{\Z}{\mathbb{Z}}
\newcommand{\E}{\mathbb{E}}
\newcommand{\PP}{\mathbb{P}}
\newcommand{\QQ}{\mathbb{Q}}
\newcommand{\A}{\mathbb{A}}
\newcommand{\D}{\mathbb{D}}

\newcommand{\one}{\mathbb{1}}
\newcommand{\bn}{\mathbf{n}}
\newcommand{\cond}{\,|\,}
\newcommand{\la}{\langle}
\newcommand{\ra}{\rangle}
\newcommand{\tree}{\Upsilon}
\newcommand{\SLE}{\textnormal{SLE}}
\newcommand{\simple}{\textnormal{simple}}
\newcommand{\dist}{\textnormal{dist}}
\newcommand{\CLE}{\textnormal{CLE}}
\newcommand{\GFF}{\textnormal{GFF}}
\newcommand{\free}{\textnormal{free}}
\newcommand{\dimH}{\textnormal{dim}}

\global\long\def\edge#1#2{\langle #1,#2 \rangle}

\newcommand{\conn}{\vartheta}
\global\long\def\ed{d}
\global\long\def\exponent{\Delta}
\global\long\def\bigrad{R}
\global\long\def\smallrad{r}
\global\long\def\sixarm#1#2#3{\LE^{n}(#1; #3, #2)}
\global\long\def\threearm#1#2#3{\LE_{\textnormal{b}}^{n}(#1; #3, #2)}

\global\long\def\path{\zeta}
\global\long\def\metric{\mathrm{dist}}

\section{Introduction}
At the  turn of the  millennium, O.~Schramm introduced  random fractal
curves in the plane which  he called ``stochastic Loewner evolutions''
($\SLE$)~\cite{SchrammScalinglimitsLERWUST,
  RohdeSchrammSLEBasicProperty}, and which have  since then been known
as  Schramm-Loewner  evolutions.  He  proved  that  these  probability
measures on  curves are the unique  ones that enjoy the  following two
properties: their law is conformally  invariant and, viewed as growth
processes (via Loewner's theory), they have the domain Markov property
--- a memorylessness  property of the growing  curve. These properties
are natural from the physics point of view, and in many cases, it has been
verified that interfaces in critical planar lattice models of statistical physics converge in  the scaling  limit to
$\SLE$  type curves;
see~\cite{SmirnovPercolationConformalInvariance, LawlerSchrammWernerLERWUST, SmirnovConformalInvariance, CamiaNewmanPercolation, SchrammSheffieldDiscreteGFF, CDCHKSConvergenceIsingSLE} for examples.

In the chordal  case, there is a  one-parameter family $(\SLE_\kappa)$
of such  curves, parameterized by a  real number $\kappa \geq  0$, which is
believed  to be  related to  universality classes  of the critical  models,
as well as to the central charges of the corresponding conformal  field theories. In this  article, we consider
several interacting  $\SLE_\kappa$ curves, multiple $\SLE$s.
We prove in Section~\ref{sec::global}
that,  when  $\kappa  \in  (0,4]$,   there  exists  a  unique  multiple
$\SLE_\kappa$ measure on families of  curves with a given connectivity
pattern, as detailed in Theorem~\ref{thm::global_unique} below.
Such measures have already been considered in many
works~\cite{KytolaMultipleSLE,    DubedatCommutationSLE,    GrahamSLE,
  KozdronLawlerMultipleSLEs, LawlerPartitionFunctionsSLE}, but we have
not  found a  conceptual approach  in the  literature, in terms of a minimal amount
of  characterizing  properties  in  the  spirit  of  Schramm's classification.

The results on convergence of a single  discrete interface to an $\SLE_\kappa$ curve
in the scaling limit are all rather involved. On the other hand, after
the  characterization of  the  multiple  $\SLE$s, it  is
relatively  straightforward to  extend  these  convergence results  to
prove  that   multiple  interfaces  also  converge   to   multiple
$\SLE_\kappa$ curves. Indeed, the relative compactness of the interfaces in a
suitable topology  can be  verified with  little effort, e.g.,  using results
in~\cite{DuminilSidoraviciusTassionContinuityPhaseTransition, KemppainenSmirnovRandomCurves},
and the  main problem  is then  to  identify  the limit  (i.e.,  to  prove that  the
subsequential  limits  are given  by  a  unique collection  of  random
curves).

As an application, we show that  the chordal  interfaces in  the
critical Ising model with  alternating boundary conditions converge to
the multiple $\SLE_\kappa$ with parameter $\kappa = 3$, in the sense detailed in
Proposition~\ref{prop::ising_multiple_cvg} and Section~\ref{subsec::ising_multiple_cvg}.
In contrast to the previous work~\cite{IzyurovIsingMultiplyConnectedDomains}
of K.~Izyurov,  we work with
the global collection of curves  and condition on the event that the interfaces form
a given connectivity pattern ---  see Figure~\ref{fig::Ising} for an
illustration. We also  identify the marginal law of one  curve in the
scaling  limit as a weighted chordal $\SLE_3$.
For the identification of the scaling limit, we use
the known convergence of a single critical Ising       interface       to        the       chordal       $\SLE_3$
\cite{CDCHKSConvergenceIsingSLE} combined with our characterization of the multiple  $\SLE_3$,
and certain technical estimates to rule out pathological behavior of the curves.

The  explicit  construction  of   global  multiple  $\SLE_\kappa$  given
in~\cite{KozdronLawlerMultipleSLEs,       LawlerPartitionFunctionsSLE,
  PeltolaWuGlobalMultipleSLEs},                 and        summarized         in
Section~\ref{sec::global}  of  the  present  article, 
fails    for   $\kappa    >   4$.    Nevertheless,   we    discuss   in
Section~\ref{sec::ising_fkperco} how, using  information from discrete
models, one could extend the classification of the multiple $\SLE_\kappa$ 
(our Theorem~\ref{thm::global_unique}) to the  range $\kappa \in (4,6]$.
More precisely, we show that the convergence of a single interface in
the critical  random-cluster model combined with
a Russo-Seymour-Welsh type (RSW)  estimate 
implies the  existence and  uniqueness of a global multiple $\SLE_\kappa$,
where $\kappa \in (4,6]$ is related to the cluster weight
$q$ 
via Equation~\eqref{eqn::relation_q_kappa}. In the  special case of the
FK-Ising model  ($q=2$), thanks to the  results of~\cite{SmirnovConformalInvarianceAnnals, ChelkakSmirnovIsing,
  CDCHKSConvergenceIsingSLE, DuminilSidoraviciusTassionContinuityPhaseTransition, KemppainenSmirnovRandomCurves}, we do  obtain the
convergence of any number of chordal interfaces to the unique multiple
$\SLE_{16/3}$ --- see Proposition~\ref{prop::fkising_alternating_cvg}.
For general  $\kappa \in (4,6)$,  this result
remains conditional on 
the  convergence  of  a  single  interface. 
The  case  $\kappa  =  6$ corresponds  to   critical  percolation,  where  the   convergence  
also follows by~\cite{SmirnovPercolationConformalInvariance, CamiaNewmanPercolation}.

\subsection{Global Multiple $\SLE$s}
\label{subsec::intro_globalmultiple}

Throughout, we  let $\Omega \subset  \C$ denote a simply  connected domain
with $2N$  distinct  points  $x_1,  \ldots,  x_{2N}  \in  \partial  \Omega$
appearing in  counterclockwise order along  the boundary on locally connected boundary segments.
We  call the $(2N+1)$-tuple $(\Omega; x_1, \ldots, x_{2N})$ a \textit{(topological) polygon}.
We consider curves in  $\Omega$ each of which connects two
points among $\{x_1, \ldots, x_{2N}\}$.  These curves can have various
planar   (i.e.,   non-crossing)   connectivities.  We   describe   the
connectivities         by        planar         pair        partitions
$\alpha  =  \{  \link{a_1}{b_1}, \ldots,  \link{a_N}{b_N}  \}$,  where
$\{a_1, b_1,\ldots,  a_N, b_N  \} =  \{1,2,\ldots,2N\}$. We  call such
$\alpha$ \textit{(planar) link patterns} and denote the set of them
by   $\LP_N$, for each $N \geq 1$.  Given   a   link  pattern   $\alpha   \in  \LP_N$   and
$\link{a}{b}       \in        \alpha$,       we        denote       by
$\alpha  \removeLink  \link{a}{b}$  the link  pattern  in  $\LP_{N-1}$
obtained by  removing $\link{a}{b}$ from $\alpha$  and then relabeling
the remaining indices so that they are the first $2(N-1)$ positive integers.

We let $X_{\simple}(\Omega; x_1,x_2)$ denote  the set of continuous simple
unparameterized  curves in  $\Omega$ connecting  $x_1$ and  $x_2$ such
that they only touch the  boundary $\partial \Omega$ in $\{x_1,x_2\}$.
When $\kappa  \in (0,4]$, the  chordal $\SLE_\kappa$ curve  belongs to
this  space   almost  surely.
Also,        when        $N        \geq       2$,        we        let
$X_{\simple}^{\alpha}(\Omega;  x_1, \ldots,  x_{2N})$  denote  the set  of
families $(\eta_1, \ldots, \eta_N)$ of pairwise disjoint curves, where
$\eta_j\in    X_{\simple}(\Omega;   x_{a_j},    x_{b_j})$   for    all
$j\in\{1,\ldots,N\}$.

\begin{definition}\label{def::global_multiple_sle}
  Let $\kappa \in (0,4]$.
  For $N\ge 2$ and for any link  pattern $\alpha \in \LP_N$, we call a
  probability         measure          on         the         families
  $(\eta_1,   \ldots,  \eta_N)\in   X_{\simple}^{\alpha}(\Omega;  x_1,
    \ldots,  x_{2N})$ \emph{a~global} $N$-$\SLE_{\kappa}$ \emph{associated  to
    $\alpha$} if, for each $j\in\{1, \ldots, N\}$, the conditional law
  of           the            curve           $\eta_j$           given 
  $\{\eta_1, \ldots, \eta_N\} \setminus \{\eta_j\}$
  is the
  chordal $\SLE_{\kappa}$  connecting $x_{a_j}$  and $x_{b_j}$  in the connected
  component of the domain 
  $\Omega  \setminus   \bigcup_{i  \neq  j}  \eta_i$   containing  the
  endpoints $x_{a_j}$ and $x_{b_j}$ of $\eta_j$ on its boundary.
\end{definition}

\begin{theorem}
  \label{thm::global_unique}
  Let $\kappa\in (0,4]$  and let $(\Omega; x_1, \ldots,  x_{2N})$ be a
  polygon with $N \geq 1$. For  any $\alpha \in \LP_N$, there exists a
  unique global $N$-$\SLE_{\kappa}$ associated to~$\alpha$.
\end{theorem}

The  existence  part  of Theorem~\ref{thm::global_unique}  is  already
known             ---             see~\cite{KozdronLawlerMultipleSLEs,
  LawlerPartitionFunctionsSLE,     PeltolaWuGlobalMultipleSLEs}.    We
briefly review the construction 
in   Section~\ref{subsec::existence}.    
J.~Miller and S.~Sheffield proved the uniqueness    part   of
Theorem~\ref{thm::global_unique} for $N=2$ in~\cite[Theorem~4.1]{MillerSheffieldIG2},
making use of   a  coupling  of the  $\SLE$  with  the Gaussian free field. 
Unfortunately, this proof does not  apply\footnote{Another proof (which might be generalizable for $N \geq 3$)
for the case of two curves recently appeared in a new appendix to~\cite{MSW}. 
However, this proof does not give exponential convergence rate of the Markov chain discussed in Remark~\ref{rem::mixing}.} 
in general for $N \geq 3$ commuting $\SLE$s. 
In the present article, we first give a different proof for the existence and uniqueness
when  $N=2$ by a Markov chain argument (in  Section~\ref{subsec::uniqueness_pair}),   and  then
generalize  the uniqueness proof      for      all     $N      \geq      3$      (in Section~\ref{subsec::uniqueness_general}).
Our proof also gives exponential convergence rate for the Markov chain.

\smallbreak

Lastly, let us note that Definition~\ref{def::global_multiple_sle} implies the following cascade  property. 
Suppose that
the collection of random curves $(\eta_1, \ldots, \eta_N)\in
  X_{\simple}^{\alpha}(\Omega; x_1, \ldots, x_{2N})$ has the law of  a global
$N$-$\SLE_{\kappa}$ associated  to the link pattern $\alpha\in\LP_N$. Assume
also that $\link{j}{j+1} \in  \alpha$ for some  $j \in \{1, \ldots,  N\}$, and
let $\eta_1$ be the curve connecting $x_j$ and $x_{j+1}$. Then, the conditional
law of the curves $(\eta_2,  \ldots, \eta_N)$ given $\eta_1$ is the global
$(N-1)$-$\SLE_{\kappa}$ associated to $\alpha\removeLink\link{j}{j+1}$.

\subsection{Multiple Interfaces in the Critical Planar Ising Model}
\label{subsec::intro_Ising}

\begin{figure} 
  \floatbox[{\capbeside\thisfloatsetup{capbesideposition={right,center},capbesidewidth=0.5\textwidth}}]{figure}[\FBwidth]
  {\caption{Simulation of  the critical  Ising model  with alternating
      boundary  conditions.  There  are  eight marked  points  on  the
      boundary of the polygon  $\Omega^{\delta}$ and therefore, four interfaces
      connect the  marked points  pairwise. We only  illustrate one
      possible connectivity of the curves  (the reader may verify that
      there are $14$ different topological possibilities).}
    \label{fig::Ising}}
  {\includegraphics[width=.4\textwidth]{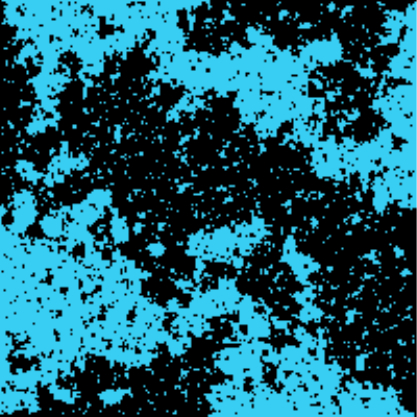}
    \qquad \quad}
\end{figure}

Next,  we consider  critical Ising  interfaces in  the scaling  limit.
Assuming   that  $\Omega$   is  bounded,   we  let   discrete  domains
$(\Omega^{\delta};  x_1^{\delta},  \ldots,  x_{2N}^{\delta})$  on  the
square  lattice   approximate  $(\Omega;  x_1,  \ldots,   x_{2N})$  as
$\delta\to 0$ (we will provide the details of the approximation scheme
in  Section~\ref{sec::ising_fkperco}), and  we  consider the  critical
Ising model (which we also define in Section~\ref{sec::ising_fkperco})
on  $\Omega^{\delta}$ with  the  following \textit{alternating  boundary
  conditions}: 
\begin{align}\label{eqn::ising_bc_alternating}
  \begin{cases}
    \oplus \text{ on }(x_{2j-1}^{\delta} \, x_{2j}^{\delta}) ,   & \quad\text{for }j\in\{1, \ldots, N\},   \\
    \ominus \text{ on }(x_{2j}^{\delta} \,  x_{2j+1}^{\delta}) , & \quad \text{for }j\in\{0,1,\ldots, N\},
  \end{cases}
\end{align}
where $(x_i^{\delta} \, x_{i+1}^{\delta})$ stands for the counterclockwise boundary arc from $x_i^{\delta}$ to $x_{i+1}^{\delta}$, 
with   the   convention  that   $x_{2N}^{\delta}=x_{0}^{\delta}$   and
$x_{2N+1}^{\delta}=x_1^{\delta}$.
With the alternating boundary conditions~\eqref{eqn::ising_bc_alternating}, 
in the configurations of the Ising model, $N$   random   interfaces
$(\eta_1^{\delta}, \ldots, \eta_N^{\delta})$ connect pairwise the $2N$ boundary
points  $x_1^{\delta},  \ldots,  x_{2N}^{\delta}$,  forming  a  planar
connectivity encoded in a link pattern $\conn^{\delta} \in \LP_N$. See Figure~\ref{fig::Ising} for an illustration.

To understand the scaling limit of the interfaces, we must specify the topology in which the convergence of the curves occurs. 
Thus, we let $X$ denote the set of planar oriented curves, that is, continuous mappings from $[0,1]$ to $\C$ modulo reparameterization.
We equip $X$ with the metric
\begin{align*} 
  \metric(\eta, \tilde{\eta}):=
  \inf_{\varphi, \tilde{\varphi}} \sup_{t\in[0,1]} |\eta(\varphi(t))-\tilde{\eta}(\tilde{\varphi}(t))| ,
\end{align*}
where  the  infimum  is   taken  over  all  increasing  homeomorphisms $\varphi, \tilde{\varphi} \colon [0,1]\to [0,1]$. 
Then,  the metric space $(X, d)$ is
complete and separable. On the space $X_{\simple}^{\alpha}(\Omega;  x_1, \ldots,
  x_{2N})$, we use the metric
\begin{align*}
  \metric ((\eta_1, \ldots, \eta_N),(\tilde{\eta}_1, \ldots, \tilde{\eta}_N) ) := \max_{1 \leq j \leq N} \metric(\eta_j, \tilde{\eta}_j) .
\end{align*}

\begin{proposition}
  \label{prop::ising_multiple_cvg}
  Let $\alpha\in   \LP_N$. Then, as $\delta \to 0$, conditionally   on   the   event
  $\{\conn^{\delta}=\alpha\}$,    the    law     of    the    collection
  $(\eta_1^{\delta},\ldots,   \eta_N^{\delta})$   of  critical   Ising
  interfaces converges weakly to the global $N$-$\SLE_3$ associated to~$\alpha$. 
  In particular, as $\delta \to 0$, the law of
  a single  curve $\eta_j^{\delta}$ in this  collection connecting two
  points  $x_a$ and  $x_b$
  converges  weakly to  a conformal  image of  the Loewner  chain with
  driving function given by Equation~\eqref{eqn::loewnerchain_purepartition} in Section~\ref{subsec:marginal} with $\kappa=3$.
\end{proposition}

We      prove       Proposition~\ref{prop::ising_multiple_cvg}      in
Section~\ref{subsec::ising_multiple_cvg}, where we also define the Ising model and discuss some of its main features.
The key ingredients in the proof are results from~\cite{DuminilSidoraviciusTassionContinuityPhaseTransition, KemppainenSmirnovRandomCurves}
for the relative compactness of the curves,  
a technical RSW  estimate~\cite{CDCHKSConvergenceIsingSLE} to rule out pathological behavior, 
and convergence of one interface~\cite{CDCHKSConvergenceIsingSLE} combined with
Theorem~\ref{thm::global_unique} for the identification of the limit.

\subsection{Multiple Interfaces in the Critical Planar FK-Ising Model}
\label{subsec::intro_FK-Ising}

In Section~\ref{sec::ising_fkperco},
we also define and discuss the  random-cluster  models, whose interfaces conjecturally converge to $\SLE_\kappa$ curves with $\kappa \in (4,6]$.
Using the discrete holomorphic observable, the convergence has been rigorously proven 
for the case of the FK-Ising model with $\kappa = 16/3$ 
for a single interface~\cite{CDCHKSConvergenceIsingSLE} and two interfaces~\cite{KemppainenSmirnovFKIsing} 
--- we provide with a proof for the general case. 
Hence, we  consider the critical FK-Ising model on $\Omega^{\delta}$ with the following alternating boundary conditions
(illustrated in Figure~\ref{fig::FK_loop_representation}):
\begin{align}\label{eq::FK_alternating}
  \begin{cases}
    \text{wired on }(x_{2j-1}^{\delta} \, x_{2j}^{\delta}) , \quad & \text{for }j\in\{1,\ldots, N\},     \\
    \text{free on }(x_{2j}^{\delta} \, x_{2j+1}^{\delta}) , \quad  & \text{for }j\in\{0,1, \ldots, N\} .
  \end{cases}
\end{align}
As in the case of the Ising model, $N$ interfaces $(\eta_1^{\delta}, \ldots, \eta_N^{\delta})$ connect
pairwise the $2N$ boundary points $x_1^{\delta}, \ldots, x_{2N}^{\delta}$,
forming a planar connectivity encoded in a link pattern $\conn^{\delta} \in \LP_N$.
However, this time the scaling limits are not simple curves, and
we need to extend the definition of a global multiple $\SLE_{\kappa}$ to include the range $\kappa \in (4,6]$.
For this, we let $X_0(\Omega; x, y)$ denote the closure of the space $X_{\simple}(\Omega; x, y)$ in the metric topology of $(X,\metric)$.
Note 
that the curves in $X_0(\Omega; x, y)$ may have multiple points but no self-crossings. In particular, for all $\kappa > 4$,
the chordal $\SLE_\kappa$ curve belongs to this space almost surely.

Then, for each $N\ge 2$ and $\alpha  =  \{  \link{a_1}{b_1}, \ldots,  \link{a_N}{b_N}  \} \in \LP_N$, we denote by $X_0^{\alpha}(\Omega; x_1, \ldots, x_{2N})$
the collection of curves $(\eta_1,\ldots, \eta_N)$ such that, for each $j\in\{1, \ldots, N\}$, we have $\eta_j\in X_0(\Omega; x_{a_j}, x_{b_j})$ and
$\eta_j$
does not disconnect any two points $x_a$, $x_b$ such that $\link{a}{b} \in \alpha$ from each other. 
Note that $X_0^{\alpha}(\Omega;  x_1, \ldots, x_{2N})$ is  not complete.
Above, the  global $N$-$\SLE_{\kappa}$ was defined  for  $\kappa  \in (0,4]$ --- we    now  extend  this
definition    to    all    $\kappa    \in    (0,8)$    by    replacing
$X_{\simple}^{\alpha}(\Omega;      x_1,\ldots,      x_{2N})$      with
$X_0^{\alpha}(\Omega; x_1, \ldots,  x_{2N})$ in Definition~\ref{def::global_multiple_sle}.
Note that this definition would actually still  formally make sense in  the range $\kappa \geq  8$, but
since  for such  values of  $\kappa$, the  $\SLE_\kappa$ process  is
described by  a Peano curve,  uniqueness of a multiple  $\SLE$ clearly
fails, as  one can  specify the common  boundaries of  the different
curves  in  an  arbitrary   way  while  preserving  the  conditional
distributions of individual curves.

\begin{proposition}\label{prop::fkising_alternating_cvg}
  Theorem~\ref{thm::global_unique} also holds for $\kappa=16/3$,
  and for any $\alpha\in   \LP_N$, as $\delta \to 0$, conditionally   on   the   event
  $\{\conn^{\delta}=\alpha\}$, the law of the collection $(\eta_1^{\delta},\ldots, \eta_N^{\delta})$ of
  critical FK-Ising interfaces converges weakly to the global $N$-$\SLE_{16/3}$ associated to $\alpha$. 
\end{proposition}

We prove Proposition~\ref{prop::fkising_alternating_cvg} in Sections~\ref{subsec::globalexistencebeyond4} and~\ref{subsec::globaluniquebeyond4}
(the proof is summarized in Section~\ref{subsec::globalexistencebeyond4}).
The relative compactness of the curves is similar as in the Ising model.
To show that the scaling limit is a global multiple $\SLE_{16/3}$,
we use the convergence of one interface~\cite{CDCHKSConvergenceIsingSLE}
combined with technical analysis using the RSW estimates~\cite{DuminilSidoraviciusTassionContinuityPhaseTransition}.
To prove the uniqueness of the limit,
we use a Markov chain argument similar to the proof of Theorem~\ref{thm::global_unique}, 
thereby also establishing the uniqueness of the global multiple $\SLE_\kappa$ for $\kappa = 16/3$. 
To this end, a priori estimates from the discrete model give us strong enough control of the curves 
(replacing the $\SLE$ analysis used for Theorem~\ref{thm::global_unique} in Section~\ref{sec::global}).

\begin{remark}
Similar arguments as presented in Sections~\ref{subsec::globalexistencebeyond4} and~\ref{subsec::globaluniquebeyond4} 
combined with the results of~\textnormal{\cite{SmirnovPercolationConformalInvariance, CamiaNewmanPercolation}}
show that
  there also exists a unique global multiple $\SLE_\kappa$ for $\kappa = 6$ with any given connectivity pattern; 
  and   Proposition~\ref{prop::fkising_alternating_cvg} holds for the critical site percolation on the triangular lattice with $\kappa=6$.
\end{remark}

\bigskip

\noindent\textbf{Acknowledgments.}
We thank C.~Garban, K.~Izyurov, A.~Karrila, A.~Kemppainen, and F.~Viklund for interesting and useful discussions.
We thank I.~Manolescu for very useful comments on the RSW estimates. 
We also thank the referee for careful comments which improved the presentation, provided an alternative proof for the existence of 
the global 2-SLEs in Section~\ref{sec::global}, and clarified Section~\ref{sec::ising_fkperco}.

V.B.   is   funded   by   the   ANR   project
ANR-16-CE40-0016. 
E.P. is funded by the Deutsche Forschungsgemeinschaft (DFG, German Research Foundation) under Germany's Excellence Strategy-EXC-2047/1-390685813.
H.W. is funded by Beijing Natural Science Foundation (JQ20001, Z180003) and Chinese Thousand Talents Plan for Young Professionals. 
During  this work, E.P.  and H.W. were supported  by the
ERC AG COMPASP, the NCCR SwissMAP,  and the Swiss~NSF. 
Part of this work was done while E.P. and H.W. visited the IHES, and
the first version of the paper was completed while E.P. and H.W. were visiting the MFO as a ``research pair".
\section{Preliminaries}
\label{sec::pre}
In this section, we give some preliminary results, for use in subsequent sections.
In Section~\ref{subsec::pre_bebl}, we discuss Brownian excursions and the Brownian loop measure.
These concepts are needed frequently in Sections~\ref{sec::pre} and~\ref{sec::global}.
In Sections~\ref{subsec::pre_sle} and~\ref{subsec::pre_sle_bdrypert}, we define the chordal $\SLE_\kappa$
and study its relationships in different domains via so-called boundary perturbation properties.
In Section~\ref{subsec::pre_sle_coupling}, we give a crucial coupling result for $\SLE$s in different domains.
This coupling is needed in the proof of Theorem~\ref{thm::global_unique} in Section~\ref{sec::global}.

\subsection{Brownian Excursions and Brownian Loop Measure}
\label{subsec::pre_bebl}

We call a polygon $(\Omega; x, y)$ with two marked points a \textit{Dobrushin domain}.
Given two boundary points $x,y \in \partial \Omega$, we denote by $(y \, x)$ the counterclockwise arc of $\partial\Omega$ from $y$ to $x$.
Also,  if $U  \subset \Omega$  is  a simply  connected subdomain  
that agrees with $\Omega$ in neighborhoods of $x$ and $y$, we say that $U$ is a \textit{Dobrushin subdomain} of $\Omega$.
For a Dobrushin domain $(\Omega; x, y)$,
the \textit{Brownian excursion measure} $\nu(\Omega; (y\, x))$
is a conformally invariant measure on Brownian excursions in 
$\Omega$ with their two endpoints on the arc $(y\, x)$
--- see~\cite[Section 3]{LawlerWernerBrownianLoopsoup} for definitions.
It is a $\sigma$-finite infinite measure, with the following restriction property:
for any Dobrushin subdomain
$U \subset \Omega$ that agrees with $\Omega$ in a neighborhood of the arc $(y\, x)$, we have
\begin{align}\label{eqn::brownianexcursion_restriction}
    \nu(\Omega; (y\, x)) [ \, \cdot \, \one_{\{e \subset U\}} ] = \nu(U; (y\, x))[ \, \cdot \, ].
\end{align}
For $\xi\geq 0$, we call a Poisson point process with intensity $\xi \nu(\Omega; (y\, x))$ a \textit{Brownian excursion soup}.


Whenever $x$ and $y$ lie on sufficiently regular boundary segments of $\Omega$, 
we define the \textit{boundary Poisson kernel $H_{\Omega}(x,y)$}
as the unique function which in the upper-half plane $\HH=\{z\in\C \colon \im{z}>0\}$ is given by
\begin{align*} 
    H_{\HH}(x,y) = |y-x|^{-2} , \qquad x,y \in \R
\end{align*}
and which in $\Omega$ is defined via conformal covariance:
for any conformal map $\varphi \colon \Omega \to \varphi(\Omega)$, we have
\begin{align}\label{eqn::poisson_cov}
    H_{\Omega}(x,y) = |\varphi'(x)\varphi'(y)| H_{\varphi(\Omega)}(\varphi(x), \varphi(y)) ,
\end{align}
and in particular,
$H_{\Omega}(x,y) := |\varphi'(x)\varphi'(y)| H_{\HH}(\varphi(x), \varphi(y))$, with $\varphi \colon \Omega \to \HH$.

\begin{lemma} \label{lem::poissonkernel_mono}
    Let $(\Omega; x, y)$ be a 
    Dobrushin domain with $x, y$ on 
    sufficiently regular boundary segments.
    Let $U, V \subset \Omega$ be two Dobrushin subdomains that agree with $\Omega$
    in a neighborhood of the arc $(y\, x)$. Then we have
    \begin{align}\label{eqn::poissonkernel_mono1}
        H_{\Omega}(x,y) \; & \ge H_{U}(x,y) ,            \\
        \label{eqn::poissonkernel_mono2}
        H_{\Omega}(x,y) \; H_{U\cap V}(x,y)
        \;                 & \ge H_{U}(x,y) \; H_V(x,y).
    \end{align}
\end{lemma}
\begin{proof}
    The inequality~\eqref{eqn::poissonkernel_mono1} follows from~\eqref{eqn::poisson_cov}.
    To prove~\eqref{eqn::poissonkernel_mono2},
    let $\PE$ be a Brownian excursion soup with intensity $\nu(\Omega; (y\, x))$.
    The union of excursions in $\PE$ satisfies the so-called one-sided restriction property (see, e.g.,~\cite[Theorem~8]{WernerConformalRestrictionRelated}),
    which implies that
    $ \PP[e\subset U\; \forall \; e\in\PE]=|\varphi'(x)\varphi'(y)|$,
    where $\varphi$ is any conformal map from $U$ onto $\Omega$ fixing $x$ and $y$.
    Combining with~\eqref{eqn::poisson_cov}, we obtain
    \begin{align*} 
        \PP[e\subset U\; \forall \; e\in\PE]=\frac{H_{U}(x,y)}{H_{\Omega}(x,y)} .
    \end{align*}
    Now, denote by $\PE_V$ the collection of excursions in $\PE$ that are contained in $V$. By \eqref{eqn::brownianexcursion_restriction},
    we know that $\PE_V$ is a Brownian excursion soup with intensity $\nu(V; (y\, x))$.
    The property~\eqref{eqn::poissonkernel_mono2} now follows from
    \begin{align*}
    \frac{H_{U\cap V}(x,y)}{H_V(x,y)}=\PP[e\subset U\; \forall \; e\in \PE_V]\ge \PP[e\subset U\; \forall \; e\in \PE]=\frac{H_{U}(x,y)}{H_{\Omega}(x,y)}. 
    \end{align*}
This  concludes the proof.
\end{proof}

The \textit{Brownian loop measure} $\mu(\Omega)$ is a conformally invariant measure on unrooted Brownian loops
in $\Omega$ ---
see, e.g.,~\cite[Sections~3~and~4]{LawlerWernerBrownianLoopsoup} for definitions.
It is a $\sigma$-finite infinite measure, which
has the following restriction property: for any subdomain $U \subset \Omega$, we have
\begin{align*}
    \mu(\Omega)[ \, \cdot  \, \one_{\{\ell \subset U\}}]=\mu(U)[ \, \cdot \, ] .
\end{align*}
For $\xi\geq 0$, we call a Poisson point process with intensity $\xi \mu(\Omega)$ a \textit{Brownian loop soup}.
This notion will be needed in Section~\ref{subsec::pre_sle_coupling}.

Given 
two disjoint subsets $V_1, V_2 \subset \Omega$, we
denote by $\mu(\Omega; V_1, V_2)$ the Brownian loop measure of loops in $\Omega$ that intersect both
$V_1$ and $V_2$. In other words,
\begin{align*}
 \mu(\Omega; V_1, V_2) := \mu\{\ell \; \colon \; \ell \subset\Omega, \; \ell \cap V_1 \neq \emptyset, \; \ell \cap V_2 \neq \emptyset\} . 
\end{align*}
If $V_1, V_2$ are at positive distance from each other, both of them are closed, and at least one of them is compact,
then we have $0 \leq \mu(\Omega; V_1, V_2) < \infty$.
Furthermore, the measure $\mu(\Omega; V_1, V_2)$ is conformally invariant:
we have
$\mu(\varphi(\Omega); \varphi(V_1), \varphi(V_2)) = \mu(\Omega; V_1, V_2)$
for any conformal map $\varphi \colon \Omega \to f(\Omega)$.

For $n$ disjoint subsets $V_1, \ldots, V_n$ of $\Omega$, we denote by $\mu(\Omega; V_1, \ldots, V_n)$ the Brownian loop measure of loops in $\Omega$
that intersect all of $V_1, \ldots, V_n$. Again, provided that $V_j$ are closed and at least one of them is compact, $\mu(\Omega; V_1, \ldots, V_n)$ is finite.
This quantity will be needed in Section~\ref{sec::global}.

\subsection{Loewner Chains and the Schramm-Loewner Evolution}
\label{subsec::pre_sle}

An \textit{$\HH$-hull} is a compact subset $K$ of $\overline{\HH}$ such that $\HH\setminus K$ is simply connected.
Riemann's mapping theorem implies that
for any hull $K$, there exists a unique conformal map $g_K$ from $\HH\setminus K$ onto $\HH$ such that $\lim_{z\to\infty}|g_K(z)-z|=0$.
Such a map $g_K$ is called the conformal map from $\HH\setminus K$ onto $\HH$ \textit{normalized at} $\infty$.
By standard estimates of conformal maps, the derivative of this map satisfies
\begin{align*} 
    0 < g_K'(x) \leq 1 \qquad \text{for all } x \in \R \setminus K .
\end{align*}
In fact, this derivative can be viewed as the probability that the Brownian excursion in $\HH$ from $x$ to $\infty$ avoids the hull $K$
--- see~\cite{LawlerSchrammWernerConformalRestriction}.

Consider a family of conformal maps $(g_{t}, t\ge 0)$ which solve the  Loewner equation: for each $z\in\mathbb{H}$,
\[\partial_{t}{g}_{t}(z)=\frac{2}{g_{t}(z)-W_{t}} \qquad \text{and} \qquad g_{0}(z)=z,\]
where $(W_t, t\ge 0)$ is some real-valued continuous function, 
called the \textit{driving function}.
Also, denote $K_{t}:=\overline{\{z\in\mathbb{H}: T_{z}\le t\}}$, where
\begin{align*}
    T_z := \sup \big\{t\ge 0: \inf_{s\in[0,t]}|g_{s}(z)-W_{s}|>0 \big\}
\end{align*}
is the \textit{swallowing time} of the point $z$.
Then, $g_{t}$ is the unique conformal map from $H_{t}:=\mathbb{H}\setminus K_{t}$ onto $\mathbb{H}$ normalized at $\infty$.
The collection of $\HH$-hulls $(K_{t}, t\ge 0)$ associated with such maps is called a \textit{Loewner chain}.
We say that $(K_t, t\ge 0)$ is generated by the continuous curve $(\gamma(t), t\ge 0)$ if,
for any $t \geq 0$, the unbounded connected
component of $\HH\setminus\gamma[0,t]$ coincides with $H_t=\HH\setminus K_t$.

\smallbreak

In this article, we are concerned with particular hulls generated by curves. For $\kappa\ge 0$,
the random Loewner chain $(K_{t}, t\ge 0)$ driven by $W_t=\sqrt{\kappa}B_t$, where $(B_t, t\ge 0)$ is
a standard Brownian motion, is called the (chordal) \textit{Schramm-Loewner
    Evolution}, or $\SLE_{\kappa}$, in $\HH$ from $0$ to $\infty$.
S.~Rohde and O.~Schramm proved in~\cite{RohdeSchrammSLEBasicProperty} that
this Loewner chain is almost surely generated by a continuous transient curve $\gamma$, with
$|\gamma(t)| \to \infty$ as $t\to\infty$,
the $\SLE_\kappa$ curve.
This random curve exhibits the following phase transitions in the parameter $\kappa$: when $\kappa\in [0,4]$, it is a simple curve;
whereas when $\kappa > 4$, it has self-touchings, being space-filling if $\kappa\ge 8$.
The law of the $\SLE_\kappa$ curve is a probability measure on the space $X_0(\HH; 0,\infty)$, and we denote it by $\PP(\HH; 0,\infty)$.

By conformal  invariance, we can define  the $\SLE_\kappa$ probability
measure  $\PP(\Omega; x,y)$  in any  simply connected  domain $\Omega$
with  two marked  boundary points  $x,y \in  \partial \Omega$  (around
which $\partial  \Omega$ is  locally connected)  via pushforward  of a
conformal  map: if  $\gamma  \sim \PP(\HH;  0,\infty)$,  then we  have
$\varphi(\gamma)       \sim       \PP(\Omega;       x,y)$,       where
$\varphi  \colon  \HH \to  \Omega$  is  any  conformal map  such  that
$\varphi(0)=x$  and $\varphi(\infty)=y$.  In fact,  by the  results of
O.~Schramm~\cite{SchrammScalinglimitsLERWUST},                     the
$\left( \SLE_\kappa \right)_{\kappa \geq  0}$ are the only probability
measures on curves $\gamma  \in X_0(\Omega; x,y)$ satisfying conformal
invariance and the following domain  Markov property: given an initial
segment     $\gamma[0,\tau]$     of    the     $\SLE_\kappa$     curve
$\gamma  \sim \PP(\Omega;  x,y)$ up  to  a stopping  time $\tau$,  the
conditional law  of the  remaining piece $\gamma[\tau,\infty)$  is the
            law   $\PP(\Omega   \setminus    K_\tau;   \gamma(\tau),y)$   of   the
        $\SLE_\kappa$ curve in  the complement of  the hull $K_\tau$ of  the initial
            segment  from the tip $\gamma(\tau)$ to~$y$.

            \subsection{Boundary Perturbation for $\SLE$s}
            \label{subsec::pre_sle_bdrypert}

            The chordal $\SLE_\kappa$ curve $\gamma \sim \PP(\Omega; x, y)$ has a natural boundary perturbation property,
            where its law in a Dobrushin subdomain of $\Omega$ is given by weighting by a factor involving the Brownian loop measure
            and the boundary Poisson kernel.
            More precisely, when $\kappa\in (0,4]$, the $\SLE_\kappa$ is a simple curve only touching the boundary at its endpoints,
and its law in the subdomain is absolutely continuous with respect to its law in $\Omega$, as we state in the next Lemma~\ref{lem::sle_boundary_perturbation}.
However, for $\kappa > 4$, we cannot have such an absolute continuity property,
because the $\SLE_{\kappa}$ has a positive chance to hit the boundary of $\Omega$.
Nevertheless, in Lemma~\ref{lem::boundary_perturbation_beyond4} we show that if we restrict the two processes in a smaller domain,
then we retain the absolute continuity for $\kappa\in (4,8)$.

\smallbreak

Throughout this article, we use the following real parameters, depending on $\kappa>0$:
\begin{align} \label{eq: alpha and central charge}
    h = \frac{6-\kappa}{2\kappa} \qquad \qquad \text{and} \qquad \qquad  c = \frac{(3\kappa-8)(6-\kappa)}{2\kappa}.
\end{align}

\begin{lemma}\label{lem::sle_boundary_perturbation}
    Let $\kappa\in (0,4]$.
    Let $(\Omega; x, y)$ be a Dobrushin domain and $U\subset\Omega$ a Dobrushin subdomain.
    Then, the $\SLE_\kappa$ in $U$ connecting $x$ and $y$
    is absolutely continuous with respect to the $\SLE_\kappa$ in $\Omega$ connecting $x$ and $y$, with
    Radon-Nikodym derivative given by
    \begin{align*}
        \frac{\ud \PP(U; x,y)}{\ud \PP(\Omega; x,y)}{(\gamma)}
        = \left(\frac{H_{\Omega}(x,y)}{H_{U}(x,y)}\right)^h
        \one_{\{\gamma\subset U\}} \exp\big(c \mu(\Omega; \gamma, \Omega\setminus U)\big) . 
    \end{align*}
\end{lemma}
\begin{proof}
    See~\cite[Section 5]{LawlerSchrammWernerConformalRestriction} and~\cite[Proposition~3.1]{KozdronLawlerMultipleSLEs}.
\end{proof}

The next lemma is a consequence of results in~\cite{LawlerSchrammWernerConformalRestriction, LawlerWernerBrownianLoopsoup}. We briefly summarize the proof.

\begin{lemma}\label{lem::boundary_perturbation_beyond4}
    Let $\kappa\in (4,8)$.
    Let $(\Omega; x, y)$ be a Dobrushin domain. Let $\Omega^L\subset U\subset\Omega$ be
    Dobrushin subdomains
    such that
    $\Omega^L$ and $\Omega$ agree in a neighborhood of the arc $(y\, x)$ and
    $\dist(\Omega^L, \Omega\setminus U)>0$.
    Then, we have
    \begin{align*}
        \one_{\{\gamma\subset\Omega^L\}}\frac{\ud \PP(U; x, y)}{\ud \PP(\Omega; x, y)}{(\gamma)}
        = \left(\frac{H_{\Omega}(x,y)}{H_U(x,y)}\right)^h \one_{\{\gamma\subset\Omega^L\}}
        \exp\big(c\mu(\Omega; \gamma, \Omega\setminus U)\big). 
    \end{align*}
\end{lemma}

\begin{proof}
    By conformal invariance, we may assume that the domain under consideration
    is $(\Omega; x, y) = (\HH; 0,\infty)$. Let
    $\gamma \sim \PP(\HH; 0,\infty)$, let $(W_t, t\ge 0)$ be its driving function,
    and $(g_t, t\ge 0)$ the corresponding conformal maps. Let $\varphi$ be the conformal map from $U$ onto $\HH$ normalized at $\infty$.
    On the event $\{\gamma\subset\Omega^L\}$, 
    define $T$ to be the first time when $\gamma$ disconnects $\HH\setminus U$ from $\infty$.

    Denote by $K_t$ the hull of $\gamma[0,t]$.
    For $t<T$, let $\tilde{g}_t$ be the conformal map from $\HH\setminus\varphi(K_t)$ onto $\HH$, and 
    let $\varphi_t$ be the conformal map from $g_t(U \setminus K_t)$ onto $\HH$, both normalized at $\infty$.
    Then we have $\tilde{g}_t\circ\varphi=\varphi_t\circ g_t$.
    Now we define, for $t<T$,
    \[M_t := \varphi_t'(W_t)^h\exp\left(-c\int_0^t\frac{S\varphi_s(W_s)}{6} \ud s \right),\]
    where $Sf$ is the Schwarzian derivative\footnote{The Schwarzian derivative
        of 
        $f$ is defined by
        $Sf(z):=\frac{f'''(z)}{f'(z)}-\frac{3f''(z)^2}{2f'(z)^2}$.}.
    It was proved in~\cite[Proposition~5.3]{LawlerSchrammWernerConformalRestriction} that 
    $M_t$ is a local martingale.
    Furthermore, using It\^o's formula, one can show that
    the law of $\gamma$ weighted by $M_t$ is
    $\PP(U; 0,\infty)$ up to time $t$.
    Also, it follows from~\cite[Proposition~5.22]{LawlerConformallyInvariantProcesses}
    (see also~\cite[Section~7]{LawlerWernerBrownianLoopsoup}) that
    \[-\int_0^t\frac{S\varphi_s(W_s)}{6} \ud s = \mu(\HH; \gamma[0,t], \HH\setminus U).\]

    Now, on the event $\{\gamma\subset\Omega^L\}$, there exists a constant $\eps=\eps(\HH, \Omega^L, U)>0$ such that for $t<T$,
    we have $\eps \leq \varphi_t'(W_t) \leq 1$.
    When $\kappa\in (4,6]$, we have $h\ge 0$ and $c\ge 0$, and thus, on the event $\{\gamma\subset\Omega^L\}$,
    we have $M_t\le \exp(c\mu(\HH; \Omega^L, \HH\setminus U))$.
    When $\kappa\in (6,8)$, we have $h\le 0$ and $c\le 0$, and in this case, we have $M_t\le \eps^h$.
    In conclusion, in either case, $(M_t, t<T)$ is uniformly bounded on the event $\{\gamma\subset\Omega^L\}$, and
    as $t\to T$, 
    we have $\varphi_t'(W_t) \to 1$ almost surely, so
    \begin{align*}
        M_t \; \to \; M_T := \exp\big(c\mu(\HH; \gamma[0,T], \HH\setminus U)\big) , \qquad \text{as }  t\to T .
    \end{align*}
    The assertion now follows by taking into account that $M_0 = \varphi'(0)^h$ and recalling the identity~\eqref{eqn::poisson_cov}.
\end{proof}

\subsection{A Crucial Coupling Result for SLEs}
\label{subsec::pre_sle_coupling}

We finish this preliminary section with a result from~\cite{WernerWuCLEtoSLE}, which says that we can construct $\SLE$s
from the Brownian loop soup and the Brownian excursion soup. This gives a coupling of $\SLE$s in two
Dobrushin domains $U \subset \Omega$,
which will be crucial in our proof of Theorem~\ref{thm::global_unique} (for Lemma~\ref{lem::sle_positivechance_stay}
for $\kappa \in [8/3,4]$).

\smallbreak

Let $(\Omega; x,y)$ be a Dobrushin domain. Let $\PL$ be a Brownian loop soup with intensity $c \mu(\Omega)$,
and $\PE$ a Brownian excursion soup with intensity $h \nu(\Omega; (y\, x))$, where $c = c(\kappa)$ and
$h = h(\kappa)$ are defined in~\eqref{eq: alpha and central charge}
and $\kappa\in [8/3,4]$.
(Note that for $\kappa\in [8/3,4]$, we have $c \in [0,1]$ and $h \in [1/4, 5/8]$.)

We say that two loops $\ell$ and $\ell'$ in $\PL$ belong to the same cluster if there exists a finite chain of loops
$\ell_0, \ldots, \ell_n$ in $\PL$ such that $\ell_0 = \ell$, $\ell_n = \ell'$, and $\ell_{j}\cap \ell_{j-1} \neq \emptyset$ for all $j \in \{1,\ldots, n\}$.
We denote by $\overline{\LC}$ the family of all closures of the loop-clusters and
by $\Gamma$ the family of all outer boundaries of the outermost elements of $\overline{\LC}$.
Then, $\Gamma$ forms a collection of disjoint simple loops, called the $\CLE_{\kappa}$ for $\kappa\in (8/3,4]$, see~\cite{SheffieldWernerCLE}.

Finally, we define $\gamma_0$ to be the right boundary of the union of all excursions $e\in\PE$ and
$\gamma$ 
the boundary of the union of $\gamma_0$ and all loops in $\Gamma$ that it intersects,
as illustrated in Figure~\ref{fig::crucialcoupling}.

\begin{figure}[h]
    \includegraphics[width=\textwidth]{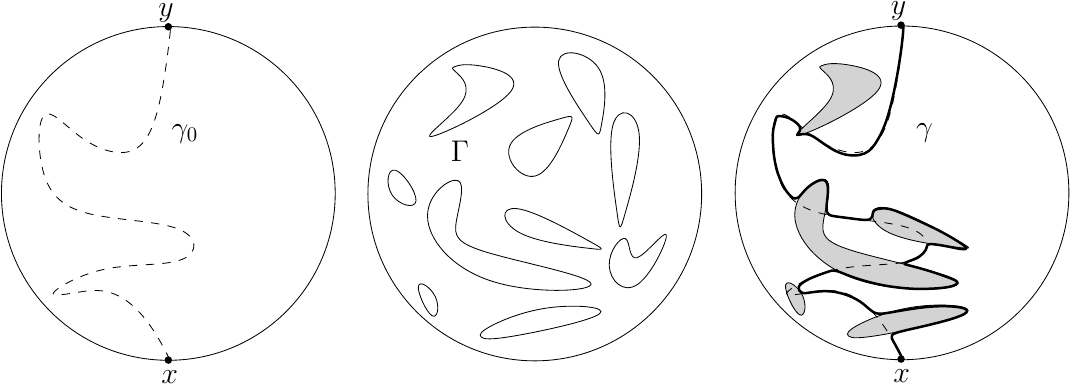}
    {\caption{In the left panel, $\gamma_0$ is the right boundary of all Brownian excursions in $\PE$.
            In the middle panel, $\Gamma$ is the family of all outer boundaries of the outermost elements of the clusters of Brownian loops in $\PL$.
            In the right panel, $\gamma$ is the right boundary of the union of $\gamma_0$ and all loops in $\Gamma$ that intersect $\gamma_0$.
            By~\cite[Theorem 1.1]{WernerWuCLEtoSLE}, we find that $\gamma\sim\PP(\Omega; x, y)$.}\label{fig::crucialcoupling}}
\end{figure}

\begin{lemma}\label{lem::sle_coupling}
    Let $\kappa\in [8/3,4]$. Let $(\Omega; x,y)$ be a Dobrushin domain
    and define $\PL$, $\PE$, $\Gamma$, $\gamma_0$, and $\gamma$ as above.
    Then, 
    $\gamma$ has the law of the $\SLE_{\kappa}$ in $\Omega$ connecting $x$ and $y$.
\end{lemma}

\begin{proof}
    When $\kappa=8/3$, the curve $\gamma$ is the same as $\gamma_0$,
    and it satisfies the so-called one-sided restriction property,
    which uniquely identifies its law with the $\SLE_{8/3}$
    by~\cite[Theorem~8.4]{LawlerSchrammWernerConformalRestriction} and~\cite[Theorem~8]{WernerConformalRestrictionRelated}.
    For $\kappa\in (8/3,4]$, the assertion was proved in~\cite[Theorem~1.1]{WernerWuCLEtoSLE}.
\end{proof}

From Lemma~\ref{lem::sle_coupling}, we see that $\SLE_{\kappa}$ curves in different domains can be coupled in the following way.
Let  $U \subset \Omega$ by a Dobrushin subdomain 
that agrees with $\Omega$ in a neighborhood of the arc $(y\, x)$.
Take $\PL$, $\PE$, $\Gamma$, $\gamma_0$, and $\gamma$ as in the above lemma.
Let $\PE_{U}$ and $\PL_{U}$ respectively be the collections of excursions in $\PE$ and loops in $\PL$ that are contained in $U$.
Define $\eta_0$ to be the right boundary of the union of all excursions $e\in\PE_{U}$,
define $\Gamma_U$ to be the collection of all outer boundaries of the outermost clusters of $\PL_{U}$,
and $\eta$ to be the right boundary of the union of $\eta_0$ and all loops in $\Gamma_U$ that it intersects.

\begin{corollary} \label{cor::sle_coupling}
    Let $\kappa\in [8/3,4]$.
    Let $(\Omega; x, y)$ be a Dobrushin domain and $U \subset \Omega$ a Dobrushin subdomain
    that agrees with $\Omega$ in a neighborhood of the arc $(y\, x)$.
    There exists a coupling $(\gamma, \eta)$ of $\gamma \sim \PP(\Omega; x,y)$ and
    $\eta \sim \PP(U; x,y)$ such that, almost surely, $\eta$ stays to the left of $\gamma$ and
    \[ \PP[\eta = \gamma] = \PP[\gamma \subset U] . \]
\end{corollary}
\begin{proof}
    Lemma~\ref{lem::sle_coupling} and the above paragraph give the sought coupling.
\end{proof}

\begin{remark}
    The coupling $(\gamma, \eta)$ of Corollary~\ref{cor::sle_coupling} is the one which maximizes the probability $\PP[\eta = \gamma]$.
\end{remark}

\section{Global Multiple SLEs}
\label{sec::global}
This section concerns the existence and uniqueness of global multiple $\SLE_\kappa$ measures for $\kappa \in (0,4]$.
Such global $N$-$\SLE$s associated to all link patterns $\alpha \in \LP_N$ and all $\kappa \in (0,4]$ have been 
constructed in~\cite{KozdronLawlerMultipleSLEs, LawlerPartitionFunctionsSLE, PeltolaWuGlobalMultipleSLEs}.
In Section~\ref{subsec::existence}, we briefly recall this construction, 
which immediately gives the existence part of Theorem~\ref{thm::global_unique}.
We prove the uniqueness part of  Theorem~\ref{thm::global_unique} in Sections~\ref{subsec::uniqueness_pair} and~\ref{subsec::uniqueness_general}.

\subsection{Construction of Global Multiple $\SLE$s for $\kappa \leq 4$}
\label{subsec::existence}

Fix a polygon $(\Omega; x_1, \ldots, x_{2N})$. 
For a link pattern $\alpha = \{ \link{a_1}{b_1}, \ldots, \link{a_N}{b_N} \} \in \LP_N$,  we let $\PP_{\alpha}$ denote
the product measure of $N$ independent chordal $\SLE_\kappa$ curves, 
\begin{align*}
    \PP_{\alpha}
    := \bigotimes_{j = 1}^{N} \PP(\Omega; x_{a_j}, x_{b_j}) ,
\end{align*}
and  $\E_{\alpha}$  the expectation with respect to $\PP_{\alpha}$.
A global $N$-$\SLE_{\kappa}$ associated to $\alpha$ can be constructed as the probability measure
$\QQ_{\alpha}^{\#} = \QQ_{\alpha}^{\#}(\Omega; x_1, \ldots, x_{2N})$
which is absolutely continuous with respect to $\PP_{\alpha}$
with explicit Radon-Nikodym derivative given in~\eqref{eqn::def_radon_alpha} below.
This formula involves a combinatorial expression $m_{\alpha}$ of Brownian loop measures, obtained by an inclusion-exclusion procedure that depends on $\alpha$.
More precisely, for a configuration $(\eta_1, \ldots, \eta_N) \in X_0^{\alpha}(\Omega; x_1, \ldots, x_{2N})$,
we define
\begin{align*} 
    m_{\alpha}(\Omega; \eta_1, \ldots, \eta_N) :=
    \sum_{\text{c.c. } \LC \text{ of } \Omega \setminus \{\eta_1, \ldots, \eta_N\} } m(\LC) ,
\end{align*}
where the sum is over all the connected components (c.c.) of the complement of the curves, and
\begin{align*}
    m(\LC)
    := & \; \sum_{\substack{i_1,i_2 \in \LB(\LC), \\ i_1\neq i_2}}
    \mu(\Omega; \eta_{i_1}, \eta_{i_2})
    - \sum_{\substack{i_1, i_2, i_3 \in \LB(\LC), \\ i_1 \neq i_2 \neq i_3 \neq i_1}}
    \mu(\Omega; \eta_{i_1}, \eta_{i_2}, \eta_{i_3})
    \\ & \;
    + \cdots + (-1)^p \mu(\Omega; \eta_{j_1}, \ldots, \eta_{j_p})
\end{align*}
is a combinatorial expression associated to the c.c. $\LC$, where
\[ \LB(\LC) := \{j\in \{1,\ldots,N\} \colon \eta_j \subset \partial\LC\} = \{j_1, \ldots, j_p\} \]
denotes the set of indices $j$ for which the curve $\eta_j$ is a part of the boundary of $\LC$.
Now, we define the probability measure $\QQ_{\alpha}^{\#}$ via
\begin{align}
    \label{eqn::def_radon_alpha}
    \frac{\ud \QQ_{\alpha}^{\#}}{\ud \PP_{\alpha}} (\eta_1,\ldots, \eta_N)
    =  & \;
    \frac{R_{\alpha}(\Omega; \eta_1, \ldots, \eta_N) }{\E_{\alpha} [R_{\alpha}(\Omega; \eta_1,\ldots, \eta_N)]}, \\
    \text{where} \qquad
    R_{\alpha}(\Omega; \eta_1, \ldots, \eta_N)
    := & \;  \one_{\{\eta_j\cap\eta_k=\emptyset \; \forall \; j\neq k\}}
    \exp \big(c m_{\alpha}(\Omega; \eta_1, \ldots, \eta_N) \big).
    \nonumber
\end{align}
By~\cite[Proposition~3.3]{PeltolaWuGlobalMultipleSLEs},
this measure satisfies the defining property of a global multiple $\SLE_\kappa$,
stated in Definition~\ref{def::global_multiple_sle}.
Also, as observed in~\cite[Equation~(3.6)]{PeltolaWuGlobalMultipleSLEs},
the expectation of $R_{\alpha}$ defines a conformally invariant and bounded function of the marked boundary points:
\begin{align*} 
    0 < f_{\alpha}(\Omega; x_1, \ldots, x_{2N}) := \E_{\alpha}[R_{\alpha}(\Omega; \eta_1, \ldots, \eta_N)] \le 1 .
\end{align*}

If $(\Omega; x_1, \ldots, x_{2N})$ is a polygon and $U\subset\Omega$ a simply connected subdomain 
that agrees with $\Omega$ in neighborhoods of $x_1, \ldots, x_{2N}$, we say that $U$ is a \textit{sub-polygon} of $\Omega$.
When the marked points $x_1, \ldots, x_{2N}$ lie on sufficiently regular
boundary segments of $\Omega$,
we may define, for all $\alpha \in \LP_N$, the functions
\begin{align}\label{eqn::purepartition_alpha_def}
    \PartF_{\alpha}(\Omega; x_1, \ldots, x_{2N}) :=
    f_{\alpha}(\Omega; x_1, \ldots , x_{2N}) \; \prod_{\link{a}{b} \in \alpha} H_{\Omega}(x_{a}, x_{b})^{h},
\end{align}
where
$H_{\Omega}$ is the boundary Poisson kernel introduced in Section~\ref{subsec::pre_bebl}.
Since $0<f_{\alpha}\le 1$, we see that
\begin{align}\label{eqn::purepartition_bounded}
    0<\PartF_{\alpha}(\Omega; x_1,\ldots, x_{2N})\le \prod_{\link{a}{b} \in \alpha} H_{\Omega}(x_{a}, x_{b})^h.
\end{align}
The functions $\PartF_{\alpha}$ are called \textit{pure partition functions} of multiple $\SLE$s.
Explicit formulas for them have been obtained when
$\kappa=2$~\cite[Theorem~4.1]{KarrilaKytolaPeltolaCorrelationsLERWUST}
and $\kappa=4$~\cite[Theorem~1.5]{PeltolaWuGlobalMultipleSLEs}.
For other values of $\kappa \in (0,8)$,
formulas in (complicated) integral form have been found
in~\cite{FloresKlebanSolutionSpacePDE4, KytolaPeltolaPurePartitionSLE}.

\smallbreak

The multiple $\SLE$ probability measure $\QQ^{\#}_{\alpha}$ has a useful boundary perturbation property.
It serves as an analogue of Lemma~\ref{lem::sle_boundary_perturbation}
in our proof of Theorem~\ref{thm::global_unique}.

\begin{proposition}\label{prop::multiplesle_boundary_perturbation}
    \textnormal{\cite[Proposition~3.4]{PeltolaWuGlobalMultipleSLEs}}
    Let $\kappa \in (0,4]$ be fixed, and let $(\Omega; x_1, \ldots, x_{2N})$ be a polygon and $U \subset \Omega$ a sub-polygon.
    Then, the probability measure $\QQ_{\alpha}^{\#}(U; x_1, \ldots, x_{2N})$ is absolutely continuous with respect to
    $\QQ_{\alpha}^{\#}(\Omega; x_1, \ldots, x_{2N})$, with Radon-Nikodym derivative
    \begin{align*}
        \; & \frac{\ud \QQ_{\alpha}^{\#}(U; x_1, \ldots, x_{2N})}{\ud \QQ_{\alpha}^{\#}(\Omega; x_1, \ldots, x_{2N})}
        (\eta_1, \ldots, \eta_N)
   =     
        \frac{\PartF_{\alpha}(\Omega; x_1, \ldots, x_{2N})}{\PartF_{\alpha}(U; x_1, \ldots, x_{2N})}
        \; \one_{\{\eta_j\subset U \; \forall \; j\}} \; \exp \bigg(c\mu \Big(\Omega; \Omega\setminus U, \bigcup_{j=1}^N \eta_j \Big) \bigg).
    \end{align*}
    Moreover, if $\kappa\le 8/3$ and $x_1, \ldots, x_{2N}$ lie on sufficiently regular
    boundary segments of $\Omega$, then
    we have
    \begin{align} \label{eqn::partitionfunction_mono}
        \PartF_{\alpha}(\Omega; x_1, \ldots, x_{2N}) \geq \PartF_{\alpha}(U; x_1, \ldots, x_{2N}).
    \end{align}
\end{proposition}


\subsection{Uniqueness for a Pair of Commuting $\SLE$s}
\label{subsec::uniqueness_pair}

Next, we prove that the global $2$-$\SLE_\kappa$ measures are unique.
This result was proved by J.~Miller and S.~Sheffield~\cite[Theorem~4.1]{MillerSheffieldIG2} using
a coupling of the $\SLE$s with the Gaussian free field ($\GFF$). We present another proof not using this coupling.
Our proof also generalizes to the case of $N \geq 3 $ commuting $\SLE$ curves, whereas couplings with the $\GFF$
seem not to be useful in that case.

\smallbreak

In this section, we focus on polygons with $N=2$. We call such a polygon $(\Omega; x_1, x_2, x_3, x_4)$ a \textit{quad}.
Because the two connectivities $\alpha \in \LP_2$ of
the curves are obtained from each other by a cyclic change of labeling of the marked boundary points, 
we may without loss of generality consider global $2$-$\SLE$s associated to 
$\alpha = \{\link{1}{4}, \link{2}{3}\}$.
Hence, throughout this section we consider pairs $(\eta^L,\eta^R)$ of simple curves such that
$\eta^L\in X_0(\Omega; x^L, y^L)$ and $\eta^R\in X_0(\Omega; x^R, y^R)$, with $\eta^L\cap\eta^R=\emptyset$.
We denote the space of such pairs by $X_0(\Omega; x^L, x^R, y^R, y^L)$.
Now, a probability measure 
on these pairs $(\eta^L,\eta^R)$
of curves is a global $2$-$\SLE_\kappa$
if the conditional law of $\eta^L$ given $\eta^R$ is that of the $\SLE_{\kappa}$ connecting $x^L$ and $y^L$
in the connected component of $\Omega \setminus \eta^R$ containing $x^L$ and $y^L$ on its boundary,
and vice versa. 

\begin{proposition}\label{prop::slepair_unique}
    For any $\kappa\in (0,4]$, there exists a unique global $2$-$\SLE_\kappa$
    on $X_0(\Omega; x^L, x^R, y^R, y^L)$.
\end{proposition}

\begin{corollary} \label{cor::slepair_unique}
    Let $\kappa\in (0,4]$.
    For any $\alpha \in \LP_2$, there exists a unique global $2$-$\SLE_\kappa$ associated to $\alpha$.
\end{corollary}
\begin{proof}
    The two
    connectivities $\alpha \in \LP_2$ of the curves are obtained from each other by a cyclic change of labeling of the marked boundary points
    $x_1, x_2, x_3, x_4$. Thus, the assertion follows from Proposition~\ref{prop::slepair_unique}.
\end{proof}

We prove Proposition~\ref{prop::slepair_unique} in the end of this section, after some technical lemmas.
The idea is to show that the global $2$-$\SLE_\kappa$ is the unique stationary measure of
a Markov chain, which at each discrete time resamples one of the two curves according to its conditional law given the other one.
In fact, the existence part is already well-known
(see, e.g.,~\cite{KozdronLawlerMultipleSLEs} and Section~\ref{subsec::existence} of the present article),
so we only need to prove the uniqueness.
Nevertheless, as pointed out by the referee, our Markov chain coupling argument actually gives both the uniqueness and existence of the stationary measure,
thanks to the following special case of the Doeblin condition.

\begin{lemma}\label{lem::Doeblin}
    Let $P$ be a Markov kernel on a measurable space 
    $E$ satisfying uniform coupling in the sense that there exists $\theta\in (0,1)$ such that
    the total variation distance between images is uniformly bounded as
    \begin{equation}
        \label{eq:coupleta}
        \sup_{x,y \in E}
        \|\delta_x P - \delta_y P \|_{\mathrm{TV}} \leq \theta.
    \end{equation}
    Then, there exists a unique $P$-stationary probability measure $\PP$, and for
    every $x \in E$, the Markov chain of kernel $P$ starting at $x$ converges in
    distribution to $\PP$.
\end{lemma}

\begin{proof}
    The key consequence of the uniform coupling~\eqref{eq:coupleta} is
    that, whenever $\PP_1$ and $\PP_2$ are two probability measures on
    $E$, we have the upper bound $\|\PP_1 P - \PP_2 P \|_{\mathrm{TV}}
        \leq \theta\|\PP_1 - \PP_2 \|_{\mathrm{TV}}$. Applying this to
    two stationary measures $\PP_1$  and $\PP_2$ readily implies the
    uniqueness. Now, let $\{X_n\}$ be a Markov chain of kernel $P$
    starting from $x\in E$, and denote by $\PP_n$ the law of $X_n$,
    i.e., $\PP_n = \delta_x P^n$. Then, for all $0 \leq n \leq m$, we have
    \begin{align*}
        \| \PP_n - \PP_m \|_{\mathrm{TV}}
        = \| \delta_x P^n - \PP_{m-n} P^n \|_{\mathrm{TV}}
        \leq \theta^n \|\delta_x - \PP_{m-n} \|_{\mathrm{TV}}
        \leq \theta^n ,
    \end{align*}
    so the sequence $\{\PP_n\}$ is Cauchy for the total variation distance.
    Thus, by the completeness of the space of measures,
    it converges to a limit $\PP$ which is $P$-stationary, thus showing the existence.
\end{proof}

%

The next key Lemmas~\ref{lem::sle_positivechance_stay}
and~\ref{lem::sle_positivechance_coincide}
are needed in order to establish the uniform coupling for Lemma~\ref{lem::Doeblin}.
The first one, Lemma~\ref{lem::sle_positivechance_stay}, is crucial:
the chordal $\SLE_\kappa$ in $\Omega$ always has a uniformly positive probability of staying in a subdomain
of $\Omega$ in the following sense.


\begin{lemma}\label{lem::sle_positivechance_stay}
    Let $\kappa\in (0,4]$. Let $(\Omega; x, y)$ be a Dobrushin domain. Let $\Omega^L, U \subset \Omega$ be Dobrushin subdomains
    such that $\Omega^L$, $U$, and $\Omega$ agree in a neighborhood of the arc $(y \, x)$.
    Suppose $\eta \sim \PP(U; x, y)$.
    Then, there exists  a constant $\theta = \theta(\Omega, \Omega^L)>0$ independent of $U$
    such that $\PP[\eta\subset\Omega^L]\ge\theta$.
\end{lemma}
\begin{proof}
    We prove the lemma separately for $\kappa\in [8/3,4]$ and $\kappa\in (0,8/3]$.
    For the former case, we make use of the coupling from Section~\ref{subsec::pre_sle_coupling}.
    For the latter, technically easier case, we use properties of the Brownian loop measure
    from Section~\ref{subsec::pre_bebl} and
    the $\SLE$ boundary perturbation property from Section~\ref{subsec::pre_sle_bdrypert}.

    When $\kappa\in [8/3,4]$, we have $c\ge 0$ by~\eqref{eq: alpha and central charge}.
    Suppose $\gamma \sim \PP(\Omega; x, y)$ and
    denote by $D_{\eta}$ (resp.~$D_{\gamma}$) the connected component of $U\setminus\eta$ (resp.~$\Omega\setminus \gamma$) with $(y \, x)$ on its boundary.
    By Corollary~\ref{cor::sle_coupling},
    there exists a coupling of $\eta$ and $\gamma$
    such that $D_{\eta}\subset D_{\gamma}$. Therefore, we have $\PP[\eta\subset\Omega^L]\ge \PP[\gamma\subset\Omega^L]> 0$.
    This gives the assertion for $\kappa\in [8/3,4]$ with $\theta(\Omega, \Omega^L)=\PP[\gamma\subset\Omega^L]>0$.

    When $\kappa\in (0,8/3]$, we have $c\le 0$ by~\eqref{eq: alpha and central charge}.
    Lemma~\ref{lem::sle_boundary_perturbation} gives
    \begin{align} \label{eq: proba to stay 1}
        \PP[\eta\subset\Omega^L]
        = \left(\frac{H_{\Omega}(x,y)}{H_U(x,y)}\right)^h
        \E \big[\one_{\{\gamma\subset\Omega^L\cap U\}}\exp \big(c\mu(\Omega; \gamma, \Omega\setminus U)\big) \big] .
    \end{align}
    Note that, on the event $\{\gamma\subset\Omega^L\cap U\}$, we have
    \begin{align}
        \nonumber
          & \;  \mu(\Omega; \gamma, \Omega\setminus (\Omega^L\cap U))                                                                                              \\
        = & \; \mu(\Omega; \gamma, \Omega\setminus U)+\mu(\Omega; \gamma, \Omega\setminus\Omega^L)-\mu(\Omega; \gamma, \Omega\setminus\Omega^L, \Omega\setminus U)
        \nonumber                                                                                                                                                  \\
        = & \; \mu(\Omega; \gamma, \Omega\setminus U)+\mu(U; \gamma, U\setminus\Omega^L).
        \label{eq: proba to stay 2}
    \end{align}
    Combining~\eqref{eq: proba to stay 1} and~\eqref{eq: proba to stay 2} and
    using Lemmas~\ref{lem::poissonkernel_mono} and~\ref{lem::sle_boundary_perturbation}, we obtain
    \begin{align*}
        \PP[\eta\subset\Omega^L]
        =                                                                              & \; \left(\frac{H_{\Omega}(x,y)}{H_U(x,y)}\right)^h \E\left[\one_{\{\gamma\subset\Omega^L\cap U\}}
            \exp \big( c\mu(\Omega; \gamma, \Omega\setminus U) \big) \right]
                                                                                       &                                                                                                   
        \\
        \ge                                                                            & \; \left(\frac{H_{\Omega}(x,y)}{H_U(x,y)}\right)^h\E\left[\one_{\{\gamma\subset\Omega^L\cap U\}}
        \exp \big( c\mu(\Omega; \gamma, \Omega\setminus(\Omega^L\cap U)) \big) \right] &                                                                                                   
        \\
        =                                                                              & \; \left(\frac{H_{\Omega^L\cap U}(x,y)}{H_{U}(x,y)}\right)^h
        \ge \left(\frac{H_{\Omega^L}(x,y)}{H_{\Omega}(x,y)}\right)^h.
                                                                                       &                                                                                                   
    \end{align*}
    This gives the assertion for $\kappa\in (0,8/3]$ with the lower bound  
    $\theta(\Omega, \Omega^L)=(H_{\Omega^L}(x,y)/H_{\Omega}(x,y))^h>0$.
\end{proof}

Next, we prove that one can couple two $\SLE$s in two
Dobrushin subdomains of $\Omega$ in such a way that
their realizations agree with a uniformly positive probability.

\begin{lemma}\label{lem::sle_positivechance_coincide}
    Let $\kappa\in (0,8)$. Let $(\Omega; x, y)$ be a Dobrushin domain.
    Let $\Omega^L \subset V\subset U, \tilde{U} \subset \Omega$ be Dobrushin subdomains such that
    $\Omega^L$ and $\Omega$ agree in a neighborhood of the arc $(y \, x)$ 
    and $\dist(\Omega^L, \Omega\setminus V)>0$.
    Suppose $\eta \sim \PP(U; x, y)$ and $\tilde{\eta} \sim \PP(\tilde{U}; x, y)$.
    Then, there exists a coupling $(\eta, \tilde{\eta})$ such that $\PP[\eta=\tilde{\eta}\subset\Omega^L]\ge\theta$,
    where the constant $\theta=\theta(\Omega, \Omega^L, V)>0$ is independent of $U$ and~$\tilde{U}$.
\end{lemma}

\begin{proof}
    First, we show that there exists a constant $p_0=p_0(\Omega, \Omega^L, V)>0$,
    independent of $U$ and $\tilde{U}$,
    such that $\PP[\eta\subset\Omega^L]\ge p_0$.
    This is true for $\kappa\le 4$ by Lemma~\ref{lem::sle_positivechance_stay}, so it remains to treat the case $\kappa\in (4,8)$.
    For this, we use the $\SLE$ boundary perturbation property from Section~\ref{subsec::pre_sle_bdrypert}.

    Let $\gamma \sim \PP(\Omega;x,y)$.
    By Lemma~\ref{lem::boundary_perturbation_beyond4}, we have
    \begin{align*}
        \PP[\eta\subset\Omega^L]=\left(\frac{H_{\Omega}(x,y)}{H_U(x,y)}\right)^h\E\left[\one_{\{\gamma\subset\Omega^L\}}
            \exp \big( c\mu(\Omega; \gamma, \Omega\setminus U) \big)\right].
    \end{align*}
    When $\kappa\in (4,6]$, we have $c\ge 0$ and $h\ge 0$ by~\eqref{eq: alpha and central charge}.
    Combining this with the inequality~\eqref{eqn::poissonkernel_mono1}, we obtain
    \begin{align*}
        \PP[\eta\subset\Omega^L]\ge \PP[\gamma\subset\Omega^L] .
    \end{align*}
    On the other hand, when $\kappa\in (6,8)$, then~\eqref{eq: alpha and central charge} implies that $c\le 0$ and $h\le 0$. 
    On the event $\{\gamma\subset\Omega^L\}$, we have
    $\mu(\Omega; \gamma, \Omega\setminus U)\le \mu(\Omega; \Omega^L, \Omega\setminus V)$, so combining with~\eqref{eqn::poissonkernel_mono1}, we obtain
    \[\PP[\eta\subset\Omega^L]\ge \left(\frac{H_{\Omega}(x,y)}{H_V(x,y)}\right)^h
        \exp \big( c\mu(\Omega; \Omega^L, \Omega\setminus V)\big) \PP[\gamma\subset\Omega^L].\]
    In either case,
    we have $\PP[\eta\subset\Omega^L]\ge p_0$ with $p_0=p_0(\Omega, \Omega^L, V)>0$, independently of $U$ and $\tilde{U}$,  as claimed.

    Next, we consider the relation between the two $\SLE_\kappa$ curves $\tilde{\eta}$ and $\eta$.
    Using Lemmas~\ref{lem::sle_boundary_perturbation} and~\ref{lem::boundary_perturbation_beyond4}, we see that
    the law of $\tilde{\eta}$ restricted to $\{\tilde{\eta}\subset \Omega^L\}$ is absolutely continuous with respect to the law of $\eta$
    restricted to $\{\eta \subset \Omega^L\}$,
    and the Radon-Nikodym derivative is given by
    \[R(\eta):=\left(\frac{H_{U}(x,y)}{H_{\tilde{U}}(x,y)}\right)^h \one_{\{\eta\subset\Omega^L\}}
        \exp \big( c\mu(U; \eta, U\setminus\tilde{U}) - c\mu(\tilde{U}; \eta, \tilde{U}\setminus U) \big).\]
    Now, the monotonicity property~\eqref{eqn::poissonkernel_mono1} shows that
    \[ \frac{H_V(x,y)}{H_{\Omega}(x,y)}\le\frac{H_{U}(x,y)}{H_{\tilde{U}}(x,y)}\le \frac{H_{\Omega}(x,y)}{H_V(x,y)} .\]
    Also, because $\Omega^L \subset V \subset U, \tilde{U} \subset \Omega$,
    we see that on the event $\{\eta \subset \Omega^L\}$, we have
    \begin{align*}
        -\mu(\Omega; \Omega^L, \Omega\setminus V) \, \le \,
        \mu(U; \eta, U\setminus\tilde{U}) - \mu(\tilde{U}; \eta, \tilde{U}\setminus U) \, \le \,
        \mu(\Omega; \Omega^L, \Omega\setminus V).
    \end{align*}
    These facts imply that $R(\eta)\ge \one_{\{\eta\subset\Omega^L\}} \, \eps$, where $\eps=\eps(\Omega, \Omega^L, V)>0$
    is independent of $U$ and $\tilde{U}$.

    Now, denote the probability $\PP[\eta\subset\Omega^L]$ by $p$.
    We conclude that the total variation distance of the law of $\tilde{\eta}$ restricted to $\{\tilde{\eta}\subset\Omega^L\}$ and the law
    of $\eta$ restricted to $\{\eta\subset\Omega^L\}$ is bounded from above by
    \[\E\left[(1-R(\eta))^+\one_{\{\eta\subset\Omega^L\}}\right]\le p-p\eps.\]
    Thus, there exists a coupling $(\tilde{\eta}, \eta)$ such that $\PP[\tilde{\eta}=\eta\subset\Omega^L]\ge p\eps$.
    From the first part of the proof, we see that $p\ge p_0(\Omega, \Omega^L, V)$. This proves the asserted result.
\end{proof}

It is important that the bounds in the technical Lemmas~\ref{lem::sle_positivechance_stay} and~\ref{lem::sle_positivechance_coincide}
are uniform over the domains $U$ and $\tilde{U}$.
In~\cite[Lemma 4.2]{MillerSheffieldIG2}, the authors proved a seemingly similar result, but they only showed that there exists
a coupling $(\eta, \tilde{\eta})$ such that $\PP[\eta=\tilde{\eta}]>0$, whereas in Lemma~\ref{lem::sle_positivechance_coincide}
we proved that $\PP[\eta=\tilde{\eta}]\ge \theta$ with the constant $\theta$ uniform over $U$ and $\tilde{U}$.

Let us also emphasize that the assumption in Lemma~\ref{lem::sle_positivechance_stay} is $\Omega^L, U\subset\Omega$, while the assumption in
Lemma~\ref{lem::sle_positivechance_coincide} is $\Omega^L\subset U\subset\Omega$. Lemma~\ref{lem::sle_positivechance_stay}
is the key point in the proof of the uniqueness in Proposition~\ref{prop::slepair_unique},
as it guarantees that there is a uniformly positive probability to couple two Markov chains for \textit{any} initial values.
In order to extend the proof of Proposition~\ref{prop::slepair_unique} for the range $\kappa \in (4,8)$,
Lemma~\ref{lem::sle_positivechance_stay} has to be extended to this range.

\begin{remark}
    It is also worthwhile to discuss the optimal value of the constant $\theta$ in Lemmas~\ref{lem::sle_positivechance_stay}
    and~\ref{lem::sle_positivechance_coincide}.
    When $\kappa\in [8/3,4]$, we know this optimal value exactly: namely,
    from the proof of Lemma~\ref{lem::sle_positivechance_stay}, we see that the optimal constant $\theta=\theta(\Omega, \Omega^L)$ equals $\PP[\gamma\subset\Omega^L]$,
    the probability of the $\SLE_\kappa$ curve  $\gamma \sim \PP(\Omega; x, y)$ to stay in $\Omega^L$.
    Also, in Lemma~\ref{lem::sle_positivechance_coincide},
    if $\kappa\in [8/3,4]$, then
    we can use the coupling of Corollary~\ref{cor::sle_coupling}, which gives the optimal constant
    $\theta=\theta(\Omega, \Omega^L, V) = \PP[\gamma\subset\Omega^L]$.
    In particular, this constant does not depend on $V$, so
    Lemma~\ref{lem::sle_positivechance_coincide}
    actually holds for all $\Omega^L\subset U, \tilde{U}\subset\Omega$.
\end{remark}

\begin{figure}[ht]
    \includegraphics[width=0.6\textwidth]{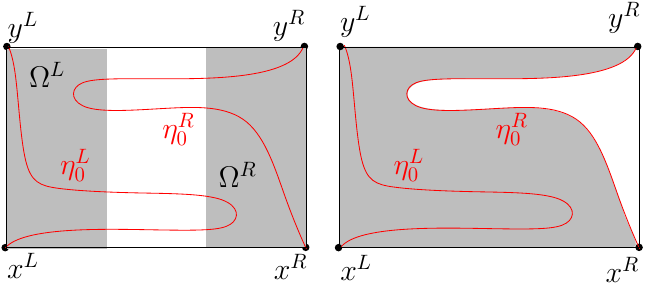}
    {\caption{\label{fig::global2}
            In the left panel, the two gray parts indicate $\Omega^L$ and $\Omega^R$ and the two red curves 
            $\eta^L_0$ and $\eta^R_0$.
            In the right panel, given $\eta_1^R=\eta_0^R$, we sample $\eta^L_1$ as the $\SLE_{\kappa}$ in the gray domain between $x^L$ and $y^L$.
            Lemma~\ref{lem::sle_positivechance_stay} guarantees that $\PP[\eta_1^L\subset\Omega^L \; | \; \eta_1^R]\ge\theta_1$.
            Then we set $\eta^L_2=\eta^L_1$, and hence, $\PP[\eta^L_2\subset\Omega^L]\ge\theta_1$.}}
\end{figure}

Now, we are ready to prove Proposition~\ref{prop::slepair_unique}.

\begin{proof}[Proof of Proposition~\ref{prop::slepair_unique}]
    By conformal invariance, it suffices to consider the domain $\Omega=[0,\ell]\times[0,1]$ 
    with marked boundary points $x^L=(0,0)$, $x^R=(\ell, 0)$, $y^R=(\ell, 1)$, $y^L=(0,1)$.
    We define a Markov chain on pairs of curves
    $(\eta^L,\eta^R) \in X_0(\Omega; x^L, x^R, y^R, y^L)$ as follows (see also Figure~\ref{fig::global2}).
    Given a configuration $(\eta^L_n, \eta^R_n)\in X_0(\Omega; x^L, x^R, y^R, y^L)$,
    we pick $i\in\{L, R\}$ uniformly and resample $\eta^i_{n+1}$
    according to the conditional law given the other curve.
    We will prove that this Markov chain has a unique stationary measure.

    Take two initial configurations $(\eta_0^L, \eta_0^R)$ and $(\tilde{\eta}_0^L, \tilde{\eta}_0^R)$.
    We will show that there exists a constant $p_0>0$, 
    independent of the initial configurations,
    and  a coupling of $(\eta_4^L, \eta_4^R)$ and $(\tilde{\eta}_4^L, \tilde{\eta}_4^R)$ such that
    \begin{align}\label{eqn::markovchain_positivechance}
        \PP[(\eta_4^L, \eta_4^R)=(\tilde{\eta}_4^L, \tilde{\eta}_4^R)]\ge p_0.
    \end{align}
    As depicted in Figure~\ref{fig::global2}, we
    denote $\Omega^L=[0, \ell/3]\times [0,1]$ and $\Omega^R=[2\ell/3, \ell]\times [0,1]$, and we denote by 
    $\theta_1 = \theta(\Omega, \Omega^L) = \theta(\Omega, \Omega^R)$ the constant obtained from Lemma~\ref{lem::sle_positivechance_stay}.
    Given an initial configuration     $(\eta^L_0, \eta^R_0) \in X_0(\Omega; x^L, x^R, y^R, y^L)$,
    we sample $\eta^L_1$ according to the conditional law and set $\eta^R_1=\eta^R_0$.
    Then, we sample $\eta^R_2$ according to the conditional law and set $\eta^L_2=\eta^L_1$. This operation has probability $1/4$.
    Knowing this sampling order, Lemma~\ref{lem::sle_positivechance_stay} gives (see Figure~\ref{fig::global2})
    \begin{align*}
        \PP[\eta^L_2 
            \subset \Omega^L] \ge \theta_1 \qquad \text{and}
        \qquad \PP[\eta^R_2 \subset \Omega^R \; | \; \eta^L_2] \ge \theta_1 .
    \end{align*}
    Thus, for any initial configurations, we have the uniform bound
    \begin{align}\label{eqn::markovchain_positivechance_aux}
        \PP\left[\eta_2^L\subset\Omega^L, \; \eta_2^R\subset\Omega^R\right]\ge \frac{1}{4} \theta_1^2 .
    \end{align}
    Now, suppose that we have two initial configurations $(\eta_0^L, \eta_0^R)$ and $(\tilde{\eta}_0^L, \tilde{\eta}_0^R)$, and
    we sample $(\eta^L_2, \eta^R_2)$ and $(\tilde{\eta}_2^L, \tilde{\eta}_2^R)$ independently.
    From~\eqref{eqn::markovchain_positivechance_aux}, we see that
    \begin{align*}
        \PP\left[\eta_2^L\subset\Omega^L, \; \tilde{\eta}_2^L \subset \Omega^L, \; \eta_2^R \subset \Omega^R, \; \tilde{\eta}_2^R \subset \Omega^R \right]
        \ge \frac{1}{16} \theta_1^4 .
    \end{align*}
    Then, given $(\eta^L_2, \eta^R_2, \tilde{\eta}^L_2, \tilde{\eta}_2^R)$,
    we resample $\eta^L_3$ and $\tilde{\eta}^L_3$ according to the conditional law and set
    $\eta^R_3=\eta^R_2$ and $\tilde{\eta}^R_3=\tilde{\eta}^R_2$.
    Lemma~\ref{lem::sle_positivechance_coincide} guarantees that there exists a coupling such that the probability
    of the event $\{\eta^L_3=\tilde{\eta}^L_3\subset\Omega^L\}$ is at least $\theta_2>0$, 
    independently of $(\eta^L_2, \eta^R_2, \tilde{\eta}^L_2, \tilde{\eta}_2^R)$ as long
    as $\{\eta_2^R, \tilde{\eta}_2^R \subset \Omega^R\}$.
    Finally, given $(\eta^L_3, \eta^R_3, \tilde{\eta}^L_3, \tilde{\eta}_3^R)$,
    we resample $\eta^R_4$ and $\tilde{\eta}^R_4$ according to the conditional law and set
    $\eta^L_4=\eta^L_3$ and $\tilde{\eta}^L_4=\tilde{\eta}^L_3$.
    Similarly, there exists a coupling such that the probability of $\{\eta^R_4=\tilde{\eta}^R_4\subset\Omega^R\}$ is
    at least $\theta_2$ as long as 
    $\{\eta^L_3, \tilde{\eta}^L_3 \subset \Omega^L\}$.
    In conclusion, there exists a coupling of $(\eta_4^L, \eta_4^R)$ and $(\tilde{\eta}_4^L, \tilde{\eta}_4^R)$ such that
    \begin{align*}
        \PP[(\eta_4^L, \eta_4^R)=(\tilde{\eta}_4^L, \tilde{\eta}_4^R)]\ge \frac{1}{64}\theta_1^4\theta_2^2.
    \end{align*}
    This implies the asserted bound~\eqref{eqn::markovchain_positivechance} with $p_0 = \frac{1}{64}\theta_1^4\theta_2^2$.

    In conclusion,
    both the existence and uniqueness of the $2$-$\SLE_\kappa$ now follow from Lemma~\ref{lem::Doeblin}
    applied to the kernel $P$ realizing four steps of the above Markov chain on $X_0(\Omega; x^L, x^R,
        y^R, y^L)$ and~\eqref{eqn::markovchain_positivechance} providing 
    the uniform coupling with $\theta=1-p_0$.
    (Furthermore, the Markov chain is mixing, see Remark~\ref{rem::mixing}.)
\end{proof}

\begin{remark}\label{rem::mixing}
    The Markov chain in the proof of Proposition~\ref{prop::slepair_unique} is mixing, that is,
    there exists a coupling between $(\eta^L_{4n}, \eta^R_{4n})$ and the global $2$-$\SLE_{\kappa}$ $(\eta^L, \eta^R)$
    so that
    \[\PP[(\eta^L_{4n}, \eta^R_{4n})\neq (\eta^L, \eta^R)]\le (1-p_0)^n. \]
\end{remark}

\bigskip

The above proof also works  when the conditional laws of $\eta^R$ and $\eta^L$ are variants of the chordal $\SLE_\kappa$.
In particular, we use
this argument for certain $\SLE$ variants in the proof of Theorem~\ref{thm::global_unique} in Section~\ref{subsec::uniqueness_general}.

\subsection{Uniqueness: General Case}
\label{subsec::uniqueness_general}

Next, we generalize our proof for the global $2$-$\SLE_\kappa$ 
to any number $N \geq 3$ of curves, in order to complete the proof of Theorem~\ref{thm::global_unique}.
Recall that, for $\alpha\in\LP_N$, we denote by $\QQ_{\alpha}^{\#}(\Omega; x_1, \ldots, x_{2N})$
the global $N$-$\SLE_\kappa$ probability measures constructed in Section~\ref{subsec::existence}.
In the general case $N \geq 3$, in order to establish the uniform coupling for Lemma~\ref{lem::Doeblin},
we use the properties of the measures $\QQ_{\alpha}^{\#}(\Omega; x_1, \ldots, x_{2N})$.
Therefore, the Markov chain argument does not yield the existence of the stationary measure.

\smallbreak

We begin by generalizing Lemma~\ref{lem::sle_positivechance_stay}. 
By symmetry, we may assume that $\link{1}{2}\in\alpha$. This lemma only uses the definition of a global multiple $\SLE_\kappa$ and
the property from Lemma~\ref{lem::sle_positivechance_stay} of the chordal $\SLE_\kappa$.

\begin{lemma}\label{lem::positivechance_stay_generalize}
    Let $\kappa\in (0,4]$.
    Let $(\Omega;x_1, \ldots, x_{2N})$ be a polygon and let $\Omega^L, U \subset \Omega$ be sub-polygons
    such that $\Omega^L$, $U$, and $\Omega$ agree in a neighborhood of the arc $(x_1 \, x_2)$.
    Also, let $(\eta_1, \ldots, \eta_N)$ be any global $N$-$\SLE_\kappa$ in $(U;x_1, \ldots, x_{2N})$
    such that $\eta_1$ is the curve connecting $x_1$ and $x_2$. Then, there exists
    a constant ${\theta=\theta(\Omega, \Omega^L)>0}$, independent of $U$,
    such that $\PP[\eta_1\subset\Omega^L]\ge\theta$.
\end{lemma}
\begin{proof}
    Denote by $\hat{U}_1$ the connected component of $U \setminus \bigcup_{j = 2}^{N} \eta_j$ with $x_1$ and $x_2$ on its boundary.
    Then, the conditional law of $\eta_1$ given $\hat{U}_1$ is the chordal $\SLE_{\kappa}$ in $\hat{U}_1$
    connecting $x_1$ and $x_2$.
    By Lemma~\ref{lem::sle_positivechance_stay}, we have
    $\PP[\eta_1\subset\Omega^L \; | \; \hat{U}_1]\ge \theta(\Omega, \Omega^L)$, independently of $\hat{U}_1$.
    Therefore, $\PP[\eta_1\subset\Omega^L]\ge \theta(\Omega, \Omega^L)$ as well.
\end{proof}

To generalize Lemma~\ref{lem::sle_positivechance_coincide}, we use the following auxiliary result, which says that all of
the curves have a positive probability to stay in a subdomain of $\Omega$, uniformly with respect to a larger subdomain. Its proof uses the explicit construction of the global $N$-$\SLE_\kappa$ measure presented in Section~\ref{subsec::existence}.

\begin{lemma}\label{lem::positivechance_stay_generalize_multiple}
    Let $\kappa\in (0,4]$.
    Let $(\Omega;x_1, \ldots, x_{2N})$ be a polygon and $\Omega^L \subset U\subset\Omega$ sub-polygons.
    Suppose $(\eta_1, \ldots, \eta_N)\sim\QQ_{\alpha}^{\#}(U; x_1, \ldots, x_{2N})$.
    Then, there exists a constant $\theta=\theta(\Omega, \Omega^L)>0$, independent of $U$,
    such that $\PP[\eta_j\subset \Omega^L \; \forall \;  j]\ge \theta$.
\end{lemma}
\begin{proof}
    We prove the lemma separately for $\kappa\in (0,8/3]$ and $\kappa\in [8/3,4]$.
    Assume first that $\kappa\in (0,8/3]$. 
    Let $(\gamma_1^L, \ldots, \gamma_N^L)$ be sampled according to $\QQ_{\alpha}^{\#}(\Omega^L; x_1, \ldots, x_{2N})$.
    By Proposition~\ref{prop::multiplesle_boundary_perturbation}, we have
    \begin{align*}
        \PP[\eta_j\subset\Omega^L \; \forall \;  j]
        = & \; \frac{\PartF_{\alpha}(\Omega^L; x_1, \ldots, x_{2N})}{\PartF_{\alpha}(U; x_1, \ldots, x_{2N})} \;
        \E \Big[ \exp \big(-c\mu \big(U; U\setminus \Omega^L,  \bigcup_{j=1}^N\gamma_j^L \big) \big) \Big] .
    \end{align*}
    Since $\kappa\le 8/3$, we have $c\le 0$ by~\eqref{eq: alpha and central charge}.
    Thus, combining with the monotonicity property~\eqref{eqn::partitionfunction_mono}, we obtain
    \begin{align*}
        \PP[\eta_j\subset\Omega^L \; \forall \;  j]
        \ge \frac{\PartF_{\alpha}(\Omega^L; x_1, \ldots, x_{2N})}{\PartF_{\alpha}(U; x_1, \ldots, x_{2N})}
        \ge \frac{\PartF_{\alpha}(\Omega^L; x_1, \ldots, x_{2N})}{\PartF_{\alpha}(\Omega; x_1, \ldots, x_{2N})} > 0,
    \end{align*}
    where the lower bound is independent of $U$, as claimed.

    Assume next that $\kappa\in [8/3,4]$. Let $(\gamma_1, \ldots, \gamma_N)\sim\QQ_{\alpha}^{\#}(\Omega; x_1, \ldots, x_{2N})$.
    By Proposition~\ref{prop::multiplesle_boundary_perturbation}, we have
    \begin{align*}
        \PP[\eta_j\subset\Omega^L \; \forall \;  j]
        = \frac{\PartF_{\alpha}(\Omega; x_1, \ldots, x_{2N})}{\PartF_{\alpha}(U; x_1, \ldots, x_{2N})}
        \; \E \Big[\one_{\{\forall j,\gamma_j\subset U \}} \exp \big(c\mu \big(\Omega; \Omega\!\setminus\!U,  \bigcup_{j=1}^N\gamma_j \big) \big) \Big] .
    \end{align*}
    Since $\kappa\in [8/3,4]$, we have $c\ge 0$ by~\eqref{eq: alpha and central charge}, so we obtain
    \begin{align*}
        \PP[\eta_j\subset\Omega^L \; \forall \;  j]
        \ge & \; \frac{\PartF_{\alpha}(\Omega; x_1, \ldots, x_{2N})}{\PartF_{\alpha}(U; x_1, \ldots, x_{2N})} \;  \PP[\gamma_j\subset\Omega^L \; \forall \;  j]                                                                        \\
        \ge & \; \frac{\PartF_{\alpha}(\Omega; x_1, \ldots, x_{2N})}{\prod_{\link{a}{b} \in \alpha} H_U(x_{a}, x_{b})^h} \;  \PP[\gamma_j\subset\Omega^L \; \forall \;  j]
            &                                                                                                                                                                         & \text{[by~\eqref{eqn::purepartition_bounded}]} \\
        \ge & \; \frac{\PartF_{\alpha}(\Omega; x_1, \ldots, x_{2N})}{\prod_{\link{a}{b} \in \alpha} H_{\Omega}(x_{a}, x_{b})^h} \; \PP[\gamma_j\subset\Omega^L \; \forall \;  j] > 0.
            &                                                                                                                                                                         & \text{[by~\eqref{eqn::poissonkernel_mono1}]}
    \end{align*}
    This gives the
    assertion for $\kappa\in [8/3,4]$ and finishes the proof.
\end{proof}

Now, we prove an analogue of Lemma~\ref{lem::sle_positivechance_coincide} for $\kappa \leq 4$ and $N \geq 3$, using the explicit construction of the global $N$-$\SLE_\kappa$ measure presented in Section~\ref{subsec::existence}.

\begin{lemma}\label{lem::positivechance_coincide_generalize}
    Let $\kappa\in (0,4]$.
    Let $(\Omega;x_1, \ldots, x_{2N})$ be a polygon, and let $\Omega^L \subset V \subset U$ and $\tilde{U}\subset\Omega$
    be sub-polygons such that $\dist(\Omega^L, \Omega\setminus V)>0$.
        Also, suppose that $(\eta_1, \ldots, \eta_N)\sim\QQ_{\alpha}^{\#}(U; x_1, \ldots, x_{2N})$ and
    $(\tilde{\eta}_1, \ldots, \tilde{\eta}_N)\sim\QQ_{\alpha}^{\#}(\tilde{U}; x_1, \ldots, x_{2N})$.
    Then, there 
    exists a coupling of $(\eta_1, \ldots, \eta_N)$ and $(\tilde{\eta}_1, \ldots, \tilde{\eta}_N)$ such that
    $\PP[\eta_j=\tilde{\eta}_j \subset \Omega^L \; \forall \;  j] \ge \theta$, where the constant $\theta=\theta(\Omega, \Omega^L, V)>0$
    is independent of $U$~and~$\tilde{U}$.
\end{lemma}

\begin{proof}
    By Proposition~\ref{prop::multiplesle_boundary_perturbation}, the law of $(\tilde{\eta}_1, \ldots, \tilde{\eta}_N)$
    restricted to $\{\tilde{\eta}_j\subset\Omega^L \; \forall \;  j\}$ is absolutely continuous with respect to the law of $(\eta_1, \ldots, \eta_N)$
    restricted to $\{\eta_j\subset\Omega^L \; \forall \;  j\}$, with Radon-Nikodym derivative
    \begin{align*}
        R(\eta_1, \ldots, \eta_N)
        = \; & \frac{\PartF_{\alpha}(U; x_1, \ldots, x_{2N})}{\PartF_{\alpha}(\tilde{U}; x_1, \ldots, x_{2N})} \;
        \one_{\{\eta_j\subset\Omega^L \; \forall \;  j\}}
        \\ \; &
        \times \exp \Big(c \mu \big(U; U\setminus\Omega^L, \bigcup_{j=1}^N\eta_j \big)
        \, - \, c \mu \big(\tilde{U}; \tilde{U}\setminus\Omega^L,  \bigcup_{j=1}^N\eta_j \big) \Big) .
    \end{align*}
    First, we will find a positive lower bound for $R(\eta_1, \ldots, \eta_N)$, separately for $\kappa\in (0,8/3]$ and $\kappa\in [8/3,4]$.
    Since $\Omega^L\subset V\subset U, \tilde{U}\subset\Omega$, on the event $\{\eta_j\subset\Omega^L \; \forall \;  j\}$, we have
    \begin{align*}
        \Big| \mu\big(U; U\setminus\Omega^L,  \bigcup_{j=1}^N\eta_j \big)
        - \mu \big(\tilde{U}; \tilde{U}\setminus\Omega^L,  \bigcup_{j=1}^N\eta_j \big)
        \Big| \le \, \mu(\Omega; \Omega\setminus V, \Omega^L).
    \end{align*}
    When $\kappa \in (0, 8/3]$, 
    we have $c\le 0$ by~\eqref{eq: alpha and central charge}.
    Thus, using the monotonicity property~\eqref{eqn::partitionfunction_mono}, we see that
    on the event $\{\eta_j\subset\Omega^L \; \forall \;  j\}$, we have
    \begin{align}\label{eqn::positivechance_coincide_smallkappa}
        R(\eta_1, \ldots, \eta_N)
        \ge \frac{\PartF_{\alpha}(\Omega^L; x_1, \ldots, x_{2N})}{\PartF_{\alpha}(\Omega; x_1, \ldots, x_{2N})}
        \; \exp \big( c\mu(\Omega; \Omega\setminus V, \Omega^L) \big) > 0 .
    \end{align}
    On the other hand, when $\kappa\in [8/3,4]$, we have $c\ge 0$ by~\eqref{eq: alpha and central charge}.
    On the event $\{\eta_j\subset\Omega^L \; \forall \;  j\}$, we have
    \begin{align*}
        R(\eta_1, \ldots, \eta_N)\ge \frac{\PartF_{\alpha}(U; x_1, \ldots, x_{2N})}{\PartF_{\alpha}(\tilde{U}; x_1, \ldots, x_{2N})}
        \; \exp \big(-c\mu(\Omega; \Omega\setminus V, \Omega^L)\big).
    \end{align*}
    Using~\eqref{eqn::purepartition_bounded} and \eqref{eqn::poissonkernel_mono1}, we estimate the denominator as
    \begin{align}\label{eqn::positivechance_coincide_aux1}
        \PartF_{\alpha}(\tilde{U}; x_1, \ldots, x_{2N})
        \le \prod_{\link{a}{b} \in \alpha} H_{\tilde{U}}(x_{a}, x_{b})^h
        \le \prod_{\link{a}{b} \in \alpha} H_{\Omega}(x_{a}, x_{b})^h ,
    \end{align}
    and using~\eqref{eqn::poissonkernel_mono1}, we estimate the numerator as
    \begin{align*}
        \PartF_{\alpha}(U; x_1, \ldots, x_{2N}) = & \;
        \prod_{\link{a}{b} \in \alpha} H_U(x_{a}, x_{b})^h \times f_{\alpha}(U; x_1, \ldots, x_{2N})
        \\ 
        \ge & \; \prod_{\link{a}{b} \in \alpha} H_{\Omega^L}(x_{a}, x_{b})^h \times f_{\alpha}(U; x_1, \ldots, x_{2N}).
    \end{align*}
    Taking the infimum over all sub-polygons $A$ such that $V\subset A\subset\Omega$, we have
    \[f_{\alpha}(U; x_1, \ldots, x_{2N}) \ge \inf_A f_{\alpha}(A; x_1, \ldots, x_{2N}) := \upsilon(\Omega, V) . \]
    We next show that this infimum is strictly positive.
    By conformal invariance of $f_{\alpha}$, we may take $\Omega = \HH$, and
    we have \[f_{\alpha}(A; x_1, \ldots, x_{2N})  = f_{\alpha}(\HH; \varphi_A(x_1), \ldots, \varphi_A(x_{2N})) > 0\]
    for any conformal map $\varphi_A \colon A \to \HH$. Now, we have
    \begin{align*}
        \upsilon(\Omega, V) = \inf_{(y_1, \ldots, y_{2N}) \in K}  f_{\alpha}(\HH; y_1, \ldots, y_{2N}) > 0 ,
    \end{align*}
    where $K$ is a compact subset  of $\R^{2N}$ such that $(\varphi_A(x_1), \ldots, \varphi_A(x_{2N})) \in K$ for all $A$. 
    Thus, we obtain
    \begin{align}\label{eqn::positivechance_coincide_aux2}
        \PartF_{\alpha}(U; x_1, \ldots, x_{2N})\ge \prod_{\link{a}{b} \in \alpha} H_{\Omega^L}(x_{a}, x_{b})^h \times \upsilon(\Omega, V)> 0.
    \end{align}
    After combining \eqref{eqn::positivechance_coincide_aux1} and \eqref{eqn::positivechance_coincide_aux2}, we finally obtain
    \begin{align}\label{eqn::positivechance_coincide_biggerkappa}
        R(\eta_1, \ldots, \eta_N)
        \ge\!\!\! \prod_{\link{a}{b} \in \alpha} \!\! \left(\frac{H_{\Omega^L}(x_{a}, x_{b}}{H_{\Omega}(x_{a}, x_{b}}\right)^h
        \upsilon(\Omega, V) \; e^{-c\mu(\Omega; \Omega\setminus V, \Omega^L)}>0.
    \end{align}
    In both estimates~\eqref{eqn::positivechance_coincide_smallkappa} and \eqref{eqn::positivechance_coincide_biggerkappa},
    we obtain a lower bound $R(\eta_1, \ldots, \eta_N)\ge \eps:=\eps(\Omega, \Omega^L, V) > 0$, independently of $U$ and $\tilde{U}$, as desired.
    This completes the first part of the proof.

    Now, denote the probability
    $\PP[\eta_j\subset\Omega^L \; \forall \;  j]$ by $p$.
    The total variation distance of the law of $(\tilde{\eta}_1, \ldots, \tilde{\eta}_N)$ restricted to $\{\tilde{\eta}_j\subset\Omega^L \; \forall \;  j\}$
    and the law of $(\eta_1, \ldots, \eta_N)$ restricted to $\{\eta_j\subset\Omega^L \; \forall \;  j\}$ is bounded from above by
    \[ \E \big[(1 - R(\eta_1, \ldots, \eta_N))^+\one_{\{\eta_j\subset\Omega^L \; \forall \;  j\}} \big] \le p(1-\eps). \]
    It follows from this observation that there exists a coupling of $(\eta_1, \ldots, \eta_N)$ and $(\tilde{\eta}_1, \ldots, \tilde{\eta}_N)$
    such that $\PP[\tilde{\eta}_j=\eta_j\subset\Omega^L \; \forall \;  j] \ge p \eps$.
    Combining this with Lemma~\ref{lem::positivechance_stay_generalize_multiple}, we obtain the asserted result.
\end{proof}

\begin{figure}[ht]
    \includegraphics[width=\textwidth]{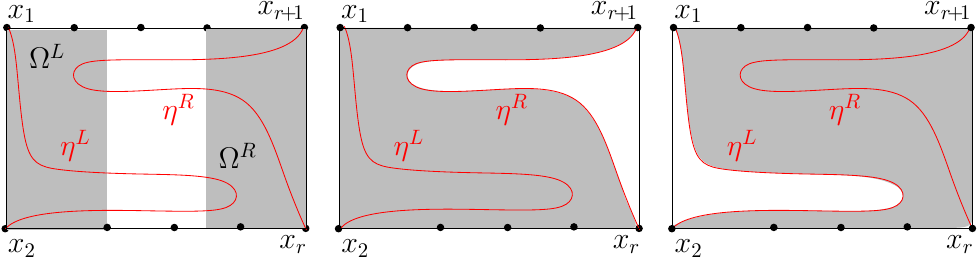}
    {\caption{\label{fig::global}
            In the left panel, the two red curves 
            are $\eta^L$ and $\eta^R$, and the two gray parts 
            are $\Omega^L$ and $\Omega^R$. In the middle panel, the gray part 
            is the domain $D^L$, and
            the marked points along the boundary of $D^L$ are $x_1, \ldots, x_{r-1}, x_{r+2}, \ldots, x_{2N}$.
            In the right panel, the gray part 
            is the domain $D^R$, and the marked points along the boundary of $D^R$ are $x_3, \ldots, x_{2N}$.}}
\end{figure}

We are now ready to conclude with the proof of Theorem~\ref{thm::global_unique}.

\begin{proof}[Proof of Theorem~\ref{thm::global_unique}]
    The existence 
    was proved in~\cite{KozdronLawlerMultipleSLEs, LawlerPartitionFunctionsSLE, PeltolaWuGlobalMultipleSLEs},
    and summarized in Section~\ref{subsec::existence}. Hence, we only need to prove the uniqueness.
    To this end, we proceed
    by induction on $N \geq 2$. The case $N=2$ is the content of Corollary~\ref{cor::slepair_unique}.
    Thus, we let $N \geq 3$ and assume that, for any link pattern
    $\beta\in\LP_{N-1}$, there exists a unique global $(N-1)$-$\SLE_{\kappa}$ associated to $\beta$.
    For $j \in \{1,\ldots, N-1\}$, we denote by $\QQ_{\beta}^{\link{a_j}{b_j}}(\Omega; x_1, \ldots, x_{2N-2})$ the marginal law of $\eta_j$
    in this global multiple $\SLE$.

    Now, let $\alpha\in \LP_N$ and suppose that $(\eta_1, \ldots, \eta_N) \in X_0^{\alpha}(\Omega; x_1, \ldots, x_{2N})$
    has the law of a global $N$-$\SLE_{\kappa}$ associated to $\alpha$.
    By symmetry, we may assume that $\link{1}{2}, \link{r}{r+1} \in \alpha$ with $r \in \{3, 4, \ldots, 2N-1\}$.
    Denote by $\eta^L$ (resp.~$\eta^R$) the curve in the collection $\{\eta_1, \ldots, \eta_N\}$ that connects $x_1$ and $x_2$ (resp.~$x_r$ and $x_{r+1}$).
    It follows from the induction hypothesis that, given $(\eta^L, \eta^R)$,
    the conditional law of the other $(N-2)$ curves  is the unique global $(N-2)$-$\SLE_{\kappa}$
    associated to $(\alpha \removeLink \link{r}{r+1}) \removeLink \link{1}{2}$ in the appropriate remaining domain
    (recalling the link removal notation).
    Thus, it is sufficient to prove the uniqueness of the joint law on the pair $(\eta^L, \eta^R)$.

    The induction hypothesis also implies that,
    given $\eta^R$ (resp.~$\eta^L$), the conditional law of the rest of the curves is
    the unique global $(N-1)$-$\SLE_{\kappa}$ associated to the link pattern $\alpha \removeLink \link{r}{r+1}$ (resp.~$\alpha \removeLink \link{1}{2}$). 
    As~illustrated in Figure~\ref{fig::global},
    we denote by $D^L$ (resp.~$D^R$) the connected component of $\Omega\setminus \eta^R$ (resp.~$\Omega\setminus \eta^L$)
    with $x_1$ and $x_{2}$ (resp.~$x_{r}$ and $x_{r+1}$) on its boundary. 
    Then, the conditional law of $\eta^L$ given $\eta^R$ is
\begin{align*}
\QQ_{\alpha\removeLink\link{r}{r+1}}^{\link{1}{2}}(D^L; x_1, \ldots, x_{r-1}, x_{r+2},\ldots, x_{2N}) 
\end{align*}
and  the conditional law of $\eta^R$ given $\eta^L$ is 
\begin{align*}
\QQ_{\alpha\removeLink\link{1}{2}}^{\link{r-2}{r-1}}(D^R; x_3, \ldots, x_{2N}) .  
\end{align*}

Now to finish, following the idea of the proof of Proposition~\ref{prop::slepair_unique}, we consider Markov chains sampling $\eta^L$ and $\eta^R$ from these conditional laws.
After replacing    
    Lemma~\ref{lem::sle_positivechance_stay} by Lemma~\ref{lem::positivechance_stay_generalize} (for $N-1$)
    and Lemma~\ref{lem::sle_positivechance_coincide} by Lemma~\ref{lem::positivechance_coincide_generalize} (also for $N-1$)
    in the proof of Proposition~\ref{prop::slepair_unique}, 
    we see that this Markov chain has a unique stationary measure which coincides with the one presented in Section~\ref{subsec::existence}.
\end{proof}

\subsection{Marginal Law}
\label{subsec:marginal}

To conclude this section, we determine the marginal law of a single curve in the global multiple $\SLE_\kappa$.
Recall that the pure partition functions $\PartF_{\alpha}$ were defined in~\eqref{eqn::purepartition_alpha_def}. We denote
\begin{align*}
    \PartF_{\alpha}(x_1, \ldots, x_{2N}):=\PartF_{\alpha}(\HH; x_1, \ldots, x_{2N}) \qquad \text{for }x_1<\cdots<x_{2N}.
\end{align*}

\begin{lemma}\label{lem::loewnerchain_purepartition}
    \textnormal{\cite[Proposition~4.9]{PeltolaWuGlobalMultipleSLEs}}
    Let $\kappa\in (0,4]$ and $\alpha\in\LP_N$. Assume that $\link{a}{b}\in\alpha$.
    Let $W_t$ be the solution to the following SDEs:
    \begin{align}\label{eqn::loewnerchain_purepartition}
        \begin{split}
            \ud W_t
            = \; & \sqrt{\kappa} \ud B_t
            +\kappa\partial_{a}\log\PartF_{\alpha} \left(V_t^1, \ldots, V_t^{a-1}, W_t, V_t^{a+1}, \ldots, V_t^{2N}\right) \ud t,\\
            \ud V_t^i = \; & \frac{2 \ud t}{V_t^i-W_t},
        \end{split}
    \end{align}
    with $W_0 = x_a$ and $V_0^i = x_i$ for $i\neq a$.
    Then, the Loewner chain driven by $W_t$ is well-defined up to the swallowing time $T_b$ of $x_b$.
    Moreover, it is almost surely generated by a continuous curve up to and including $T_b$.
    This curve has the same law as the one connecting $x_a$ and $x_b$ in the global multiple $\SLE_{\kappa}$ associated to $\alpha$
    in the polygon $(\HH; x_1, \ldots, x_{2N})$.
\end{lemma}


\section{Multiple Interfaces in Ising and Random-Cluster Models}
\label{sec::ising_fkperco}
In this final section, we give examples of discrete models whose interfaces converge in the scaling limit to multiple $\SLE$s.
More precisely, we consider the critical Ising and random-cluster models in the plane.

In Sections~\ref{subsec::fkperco}~--~\ref{subsec::globaluniquebeyond4}, we  consider interfaces in the
critical random-cluster models with alternating boundary conditions and fixing the connectivity pattern of the curves.
We show that, given the convergence of a single interface, 
multiple interfaces also have a conformally invariant scaling limit,
namely the unique global multiple $\SLE_\kappa$ with $\kappa \in (4,6]$.
Interestingly, this range of the parameter $\kappa$ is beyond
the range $(0,4]$, where global multiple $\SLE$s have been explicitly constructed and classified.
Thus, from the convergence of these discrete interfaces we
would in fact get the existence and uniqueness of the global multiple $\SLE_\kappa$ with $\kappa \in (4,6]$.
Unfortunately, the convergence of a single interface in the random-cluster model towards the chordal $\SLE_\kappa$ has only
been rigorously established for the case of $\kappa = 16/3$ --- the FK-Ising model.
This is the case appearing in Proposition~\ref{prop::fkising_alternating_cvg}, whose proof is completed in Section~\ref{subsec::globaluniquebeyond4}.
The convergence of two interfaces of the FK-Ising model was also proved in~\cite{KemppainenSmirnovFKIsing},
where the authors used a discrete holomorphic observable constructed
in~\cite{SmirnovConformalInvariance, SmirnovConformalInvarianceAnnals, ChelkakSmirnovIsing}.
In contrast, our method gives the convergence for any given number of interfaces via a global approach.

In the case of the critical Ising model with alternating boundary conditions, K.~Izyurov proved
that the collection of any number $N$ of interfaces converges to a multiple $\SLE$ process in a local sense~\cite{IzyurovIsingMultiplyConnectedDomains}.
In the present article, we condition the interfaces to forming a given connectivity pattern and prove the convergence of
the interfaces as a whole global collection of curves, which we know by Theorem~\ref{thm::global_unique} to be given by the unique global $N$-$\SLE_3$.
This is the content of Section~\ref{subsec::ising_multiple_cvg}, where we prove Proposition~\ref{prop::ising_multiple_cvg}.
We are also able to determine the marginal law of one curve in this scaling limit.
The case of two curves was considered in~\cite{WuHyperSLE}: in this case, the marginal law is
also called a hypergeometric $\SLE$.

In~\cite[Sections~5~and~6]{PeltolaWuGlobalMultipleSLEs}, the authors discussed multiple level lines of the Gaussian free field with alternating boundary data.
These level lines give rise to global multiple $\SLE_4$ curves (with any connectivity pattern).
In this particular case, the marginal law of one curve in the global multiple $\SLE_4$ degenerates to a certain $\SLE_4(\rho)$ process.
In general, however, the marginal laws of single curves in global multiple $\SLE$s are not $\SLE_\kappa(\rho)$ processes, but 
certain more general variants of the chordal $\SLE_\kappa$. We refer to~\cite[Section~3]{PeltolaWuGlobalMultipleSLEs} for more details.

\bigbreak

\textbf{Notation and terminology.}
We will use the following notions throughout.
For notational simplicity, we only consider the square lattice $\Z^2$.
Two vertices $v$ and $w$ are said to be \textit{neighbors} if their Euclidean
distance equals one, which we denote by $v \sim w$.
For a finite subgraph $\graph = (V(\graph), E(\graph))$ of $\Z^2$, we denote by $\partial \graph$ the inner boundary of $\graph$:
\begin{align*}
  \partial \graph = \{ v \in V(\graph) \, \colon \, \exists \; w \not\in V(\graph) \text{ such that }\edge{v}{w}\in E(\Z^2)\} .
\end{align*}
As an abuse  of notation,  we sometimes let  $\graph$ also
denote the simply connected domain formed  by all of the faces, edges,
and vertices of $\graph$.

In the case of the square lattice, the \textit{dual lattice} $(\Z^2)^*$ is just a translated version of $\Z^2$. More precisely,
$(\Z^2)^*$ is the dual graph of $\Z^2$: its vertex set is $(1/2, 1/2)+\Z^2$ and its edges are given by
all pairs $(v_1, v_2)$ of vertices that are neighbors. 
The vertices and edges of $(\Z^2)^*$ are called dual-vertices and dual-edges.
In particular, for each edge $e$ of $\Z^2$, we associate a dual-edge, denoted by $e^*$, that crosses $e$ in the middle.
For a 
subgraph $\graph$ of $\Z^2$, we define $\graph^*$ to be the subgraph of $(\Z^2)^*$ with edge set $E(\graph^*)=\{e^*: e\in E(\graph)\}$ and
vertex set given by the endpoints of these dual-edges.

Finally, the \textit{medial lattice} $(\Z^2)^{\diamond}$ is the graph with the centers
of edges of $\Z^2$ as the vertex set, and edges 
given by all pairs of vertices that are 
neighbors. In the case of the square lattice, the medial lattice is a rotated and rescaled version of $\Z^2$.
We identify the faces of $(\Z^2)^{\diamond}$ with the vertices of $\Z^2$ and $(\Z^2)^*$.

Now, suppose that $\graph$ is  a  finite connected  subgraph  of the  (possibly
translated, rotated, and rescaled) square lattice $\Z^2$ such that the
complement of $\graph$ is also  connected (this means that $\graph$ is
simply connected). Then, we call a  triple $(\graph; v, w)$ with $v, w
  \in  \partial \graph$ distinct boundary vertices \textit{a discrete  Dobrushin domain}.  We note
that the boundary  $\partial \graph$ is divided into  two arcs $(v \, w)$
and $(w \, v)$. More generally, given distinct boundary vertices
$v_1, \ldots, v_{2N} \in \partial  \graph$ in counterclockwise order, we call the $(2N+1)$-tuple
$(\graph;  v_1,  \ldots,  v_{2N})$  \textit{a  discrete
  polygon}. In this case, the boundary $\partial \graph$ is divided into
$2N$ arcs.

\smallbreak

In  this article,  we consider  scaling limits  of models  on
discrete lattices with mesh size tending to zero. We only consider the
following  \textit{square  lattice  approximations}, even  though  the  results
discussed  in  this  section  hold   in  a  more  general  setting  as
well, see~\cite{ChelkakSmirnovIsing}.   For   small  $\delta>0$,   we   let
$\Omega^{\delta}$ denote a finite subgraph  of the rescaled square lattice
$\delta\Z^2$.  Like $\Omega^{\delta}$,  we decorate  its vertices  and
edges with the mesh size $\delta$ as a superscript.
The definitions of the dual lattice $\Omega^{\delta}_* := (\Omega^{\delta})^*$, the medial lattice $\Omega^{\delta}_{\diamond} := (\Omega^{\delta})^{\diamond}$,
and discrete Dobrushin domains and polygons obviously extend to this context.

\begin{definition}
  Let  $(\Omega;  x_1,  \ldots,  x_{2N})$   be  a  bounded  polygon  and consider a sequence
  $( (\Omega^{\delta}; x_1^{\delta}, \ldots,  x_{2N}^{\delta}) )_{\delta > 0}$
  of  discrete polygons.  We say  that $(\Omega^{\delta};  x_1^{\delta},
    \ldots, x_{2N}^{\delta})$ converges to $(\Omega; x_1, \ldots, x_{2N})$ as $\delta \to 0$
  \emph{in  the Carath\'{e}odory  sense} if  there exist  conformal maps
  $f^{\delta}$  \textnormal{(}resp.~$f$\textnormal{)}  from  the  unit  disc  $\U=\{z\in\C  \colon
    |z|<1\}$ to $\Omega^{\delta}$ \textnormal{(}resp.~from  $\U$ to $\Omega$\textnormal{)} such that
  $f^{\delta}\to f$ uniformly on any compact subset  of $\U$, and for  all $j \in
    \{1, \ldots,2N\}$,   we   have  $\underset{\delta   \to   0}{\lim}
    (f^\delta)^{-1}(x_j^{\delta}) = f^{-1}(x_j)$.
\end{definition}

\subsection{Random-Cluster Models}
\label{subsec::fkperco}

Let $\graph = (V(\graph), E(\graph))$ be a finite subgraph of $\Z^2$.
A \textit{(percolation) configuration} $\omega=(\omega_e)_{e \in E(\graph)}$ is an element of $\{0,1\}^{E(\graph)}$.
If $\omega_e=1$, the edge $e$ is said to be \textit{open}, and otherwise, $e$ is said to be \textit{closed}. The configuration $\omega$ can be seen
as a subgraph of $\graph$ with the same set of vertices $V(\graph)$  and 
whose edges are the open edges $\{e\in E(\graph) \colon \omega_e=1\}$.
We denote by $o(\omega)$ (resp.~$c(\omega)$) the number of open (resp.~closed) edges of~$\omega$.

We are interested in the connectivity properties of the graph $\omega$. The maximal connected components of $\omega$ are called \textit{clusters}.
Two vertices $u$ and $v$ are connected by $\omega$ inside $S\subset \Z^2$ if there exists a sequence  $\{v_j \colon 0\le j\le k\}$  of vertices in $S$ such that
$v_0=u$, $v_k=v$, and each edge $\edge{v_j}{v_{j+1}}$ is open in $\omega$ for $0\le j < k$.

We may also impose to our model various boundary conditions, which can be understood as encoding how the sites are connected outside $\graph$.
A \textit{boundary condition} $\xi$ is a partition $P_1\sqcup \cdots\sqcup P_k$ of $\partial \graph$.
Two vertices are said to be \textit{wired} in $\xi$ if they belong to the same $P_j$
and \textit{free} otherwise.
We denote by $\omega^{\xi}$ the (quotient) graph obtained from the configuration $\omega$ by identifying the wired vertices together in $\xi$.

The probability measure $\smash{\phi^{\xi}_{p,q,\graph}}$ of the \textit{random-cluster model} on $\graph$ with edge-weight $p\in [0,1]$, cluster-weight $q>0$,
and boundary condition $\xi$, is defined by
\begin{align*}
  \phi^{\xi}_{p,q,\Omega}[\omega]:= \;          & \frac{p^{o(\omega)}(1-p)^{c(\omega)}q^{k(\omega^{\xi})}}{Z^{\xi}_{p,q,\Omega}} ,
    \qquad
  \text{where } \quad Z^{\xi}_{p,q,\Omega} = \sum_{\omega} p^{o(\omega)}(1-p)^{c(\omega)}q^{k(\omega^{\xi})} ,
\end{align*}
where
and $k(\omega^{\xi})$ is the number of 
connected components of the graph $\omega^{\xi}$.
For $q=1$, this model is simply the Bernoulli bond percolation.
For $q=2$, the random-cluster model is also known as the FK-Ising model, closely related to the spin Ising model.
Proposition~\ref{prop::fkising_alternating_cvg} concerns this case.

For a configuration $\xi$ on $E(\Z^2)\setminus E(\graph)$, the boundary condition induced by $\xi$ is defined as the partition
$P_1\sqcup \cdots \sqcup P_k$, where two vertices belong to the same $P_j$ if and only if there exists an open path in $\xi$ connecting them.
We identify the boundary condition induced by $\xi$ with the configuration itself, and denote the measure of the random-cluster model with such boundary conditions by
$\smash{\phi^{\xi}_{p,q,\graph}}$. As a direct consequence of these definitions, we have the following domain    Markov    property.
Suppose that $\graph \subset \graph'$ are two finite subgraphs of $\Z^2$.
Fix $p\in [0,1]$, $q>0$, and a boundary condition $\xi$ on $\partial \graph'$.
Let $X$ be a random variable, which is measurable with respect to the status of the edges in $\graph$.
Then, for all $\psi\in \{0,1\}^{E(\graph')\setminus E(\graph)}$,  we have
\begin{align*}
  \phi^{\xi}_{p,q,\graph'}\left[X \; | \; \omega_e = \psi_e  \text{ for all } e \in E(\graph')\setminus E(\graph)\right] = \phi^{\psi^{\xi}}_{p,q,\graph}[X],
\end{align*}
where $\psi^{\xi}$ is the partition on $\partial \graph$ obtained by wiring two vertices in $\partial \graph$  if they are connected in $\psi$.

\smallbreak

We define an ordering for configurations as follows.
For $\omega, \omega'\in \{0,1\}^{E(\graph)}$, we denote by $\omega\le \omega'$ if $\omega_e\le \omega'_e$ for all $e\in E(\graph)$.
An event $\LA$ depending
on the edges in $E(\graph)$ is said to be \textit{increasing} if, for any $\omega\in \LA$,
the inequality $\omega\le \omega'$ implies that $\omega'\in \LA$.
When $q\ge 1$, the following FKG inequality (positive association) holds.
Fix $p\in [0,1], q\ge 1$, and a boundary condition $\xi$ on $\partial \graph$.
Then, for any two increasing events $\LA$ and $\LB$, we have
\[\phi^{\xi}_{p,q,\graph}[\LA\cap \LB]\ge \phi^{\xi}_{p,q,\graph}[\LA]\phi^{\xi}_{p,q,\graph}[\LB].\]
Consequently, for any boundary conditions $\xi\le\psi$ and for any increasing event $\LA$, we have
\begin{align}\label{eqn::rcm_boundary_comparison}
  \phi^{\xi}_{p,q,\graph}[\LA]\le \phi^{\psi}_{p,q,\graph}[\LA].
\end{align}

We denote by $\phi^{0}_{p,q,\graph}$ the probability measure of the random-cluster model
with \textit{free boundary conditions}, where the partition $\xi$ of $\partial \graph$ consists of singletons only.
We denote by $\phi^{1}_{p,q,\graph}$ the probability measure of the random-cluster model with \textit{wired boundary conditions}, where the partition $\xi$ of $\partial\graph$ is the whole set $\partial\graph$.
In the sense of~\eqref{eqn::rcm_boundary_comparison}, $\phi^0_{p,q,\graph}$ is minimal and $\phi^1_{p,q,\graph}$ is maximal.

A configuration $\omega$ on $\graph$ can be uniquely associated to a dual configuration $\omega^*$ on the dual graph $\graph^*$,
defined by $\omega^*(e^*)=1-\omega(e)$ for all $e\in E(\graph)$. A dual-edge $e^*$ is said to be \textit{dual-open} if $\omega^*(e^*)=1$
and \textit{dual-closed} otherwise. A \textit{dual-cluster} is a connected component of $\omega^*$.
We extend the notions of dual-open paths and connectivity events in the obvious way.
Now, if $\omega$ is distributed according to $\smash{\phi^{\xi}_{p,q,\graph}}$, then $\omega^*$ is distributed according to $\smash{\phi^{\xi^*}_{p^*, q^*, \graph^*}}$, with
\[q^*=q \qquad \text{and} \qquad \frac{pp^*}{(1-p)(1-p^*)}=q.\]
Note that, at the self-dual point $p^*=p$, we have
\[ p=p_c(q):=\frac{\sqrt{q}}{1+\sqrt{q}} . \]
For this critical case
$p=p_c(q)$, we have the following Russo-Seymour-Welsh (RSW) estimate.
For a rectangle $R = [a,b] \times [c,d] \subset \Z^2$, we let $\LC_{\textnormal{hor}}(R)$ denote the event that
there exists an open path in $R$ from $\{a\} \times [c,d]$ to $\{b\}\times [c,d]$.
For the probability of this event, we have a lower bound which is uniform in the
size of the rectangle (but which depends on the shape, and is not expected to hold for $q=4$):

\begin{proposition}\label{prop::rcm_rsw_rectangle}
  \textnormal{\cite[Theorem~7]{DuminilSidoraviciusTassionContinuityPhaseTransition}}
  Let $1\le q<4$ and $u>0$, and denote by $R_n^{u}$ the rectangle $[[0,u n]]\times [[0,n]]$ for $n\ge 1$.
  Then, there exists a constant $\theta(u)>0$ such that
  \begin{align*} 
    \phi^0_{p_c(q), q, R_n^u}\left[\LC_{\textnormal{hor}}(R_n^{u})\right]\ge \theta(u) \qquad \text{for any $n\ge 1$}.
  \end{align*}
\end{proposition}


\begin{figure} [h]
  \begin{center}
    \includegraphics[width=.6\textwidth]{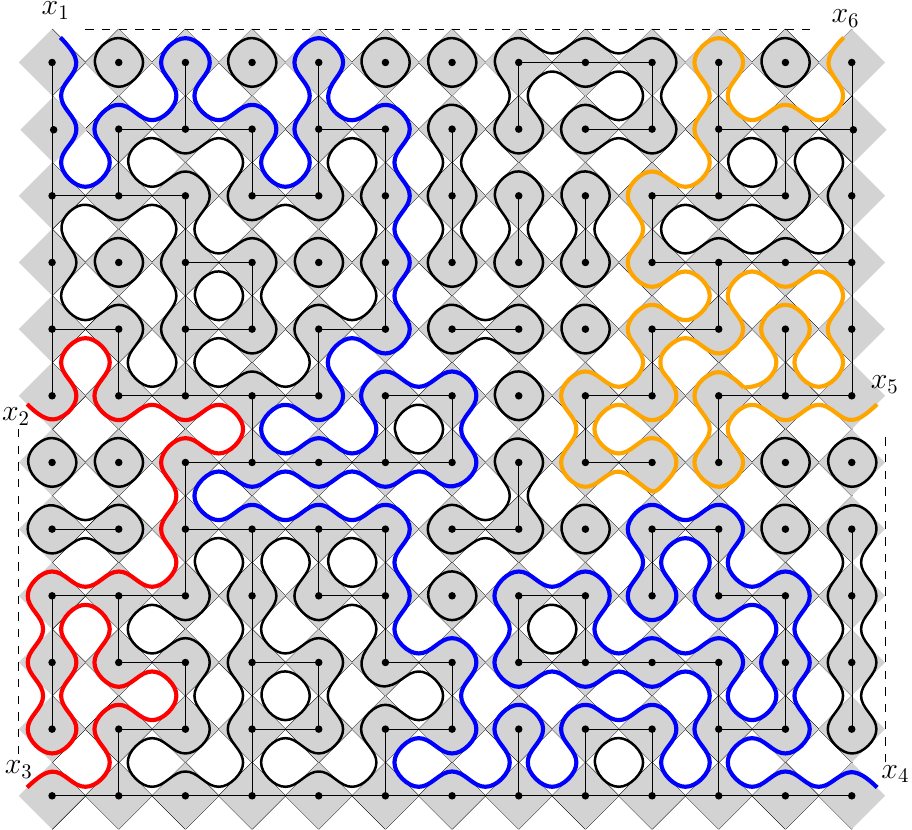}
  \end{center}
  \caption{The loop representation of a configuration of the random-cluster model on a polygon with six marked points $x_1, \ldots, x_6$ on the boundary, with
    alternating boundary conditions. There are three interfaces connecting the marked boundary points, illustrated in red, blue, and orange, respectively.}
  \label{fig::FK_loop_representation}
\end{figure}

\smallbreak

Next, we consider interfaces.
If $(\graph; u, v)$ is a discrete Dobrushin domain, in the \textit{Dobrushin boundary conditions} for the random-cluster model,
all edges along the arc $(v \, u)$ are wired and all edges along $(u \, v)$ are free. 
Then, for each vertex $w$ of the medial graph $\graph^{\diamond}$, there exists either an open edge of $\graph$ or a dual-open edge of $\graph^*$
passing through $w$. In addition, we can draw self-avoiding loops on $\graph^{\diamond}$ as follows: a loop arriving at a vertex of the medial lattice always makes
a turn of $\pm\pi/2$, so as not to cross the open or dual-open edges through this vertex.
The loop representation contains loops and the self-avoiding path connecting two vertices $u^{\diamond}$ and $v^{\diamond}$
of the medial graph $\graph^{\diamond}$ that are closest to $u$ and $v$. This curve is called the \textit{interface} (the exploration path)
of the random-cluster model.

At the critical point $p = p_c(q)$, this interface is expected to converge weakly in the scaling limit
to the chordal $\SLE_\kappa$ curve, with $\kappa$ specifically given by $q$.
The convergence has been rigorously established for the special case of $q=2$, also known as the FK-Ising
model~\cite{CDCHKSConvergenceIsingSLE}, in the topology of Section~\ref{subsec::intro_Ising}.

\begin{conjecture}\label{conj::fkperco_cvg}
  \textnormal{(See, e.g.,~\cite{SchrammICM}.)}
  Let $0\le q\le 4$ and $p=p_c(q)$. Let $(\Omega^{\delta}; x^{\delta}, y^{\delta})$ be a sequence of discrete Dobrushin domains converging to
  a Dobrushin domain $(\Omega; x, y)$ in the Carath\'eodory sense.
  Then, as $\delta\to 0$, the interface of the critical random-cluster model on $(\Omega^{\delta}; x^{\delta}, y^{\delta})$,
  with cluster weight $q$ and  Dobrushin boundary conditions,
  converges weakly to the chordal $\SLE_{\kappa}$ connecting $x$ and~$y$ with
  \begin{align}\label{eqn::relation_q_kappa}
    \kappa=\frac{4\pi}{\arccos(-\sqrt{q}/2)}.
  \end{align}
\end{conjecture}

\begin{theorem}\label{thm::fkising_cvg}
  \textnormal{\cite[Theorem~2]{CDCHKSConvergenceIsingSLE}}
  Conjecture~\ref{conj::fkperco_cvg} holds for $q=2$ and $\kappa=16/3$.
\end{theorem}

\subsection{Existence of Global Multiple $\SLE$s with $\kappa\in (4,6]$}
\label{subsec::globalexistencebeyond4}

We consider the convergence of random-cluster interfaces in the following setup.
Abusing and lightening notation, let us write $x^{\delta}$ for both $x^{\delta}$ and $(x^{\diamond})^{\delta}$
(converging to the same point $x$ as $\delta \to 0$).
Let the polygons $(\Omega^{\delta}; x_1^{\delta}, \ldots, x_{2N}^{\delta})$
converge to $(\Omega; x_1, \ldots, x_{2N})$ as $\delta\to 0$ in the Carath\'eodory sense.
Consider the critical random-cluster model on $\Omega^{\delta}$ with alternating boundary conditions~\eqref{eq::FK_alternating}.
With such boundary conditions, there are $N$ interfaces $(\eta^{\delta}_1, \ldots, \eta_N^{\delta})$
connecting pairwise the $2N$ boundary points $x_1^{\delta}, \ldots, x_{2N}^{\delta}$,
as illustrated in Figure~\ref{fig::FK_loop_representation}.
These interfaces form a planar connectivity encoded in a link pattern $\conn^{\delta} \in \LP_N$.
We  consider   the  interfaces   conditionally  on  forming   a  given
connectivity  $\conn^{\delta} =  \alpha  =  \{ \link{a_1}{b_1},  \ldots, \link{a_N}{b_N}    \}  \in \LP_N$.

Conjecturally,
conditionally on the event $\{\conn^{\delta}=\alpha\}$,
the law of the collection $(\eta^{\delta}_1, \ldots, \eta_N^{\delta})$ converges weakly as $\delta \to 0$ to a global $N$-$\SLE_{\kappa}$ associated to $\alpha$,
where $\kappa$ is determined by $q$ via~\eqref{eqn::relation_q_kappa}.
In this section, we will prove this statement for the case of $q=2$ (so $\kappa = 16/3$) --- this is the content of Proposition~\ref{prop::fkising_alternating_cvg}.
The main inputs to the proof are Theorem~\ref{thm::fkising_cvg} concerning convergence of one interface,
classification of multiple $\SLE_{16/3}$ measures analogously to Theorem~\ref{thm::global_unique},
and the RSW estimate from Proposition~\ref{prop::rcm_rsw_rectangle}.

\begin{proof}[Proof of Proposition~\ref{prop::fkising_alternating_cvg}]
  Conditionally     on      $\{\conn^{\delta}=\alpha\}$,     we     have
  \begin{align*}
    (\eta_1^{\delta},            \ldots,            \eta_N^{\delta})\in
    X_0^{\alpha}(\Omega^{\delta};          x_1^{\delta},         \ldots,
    x_{2N}^{\delta}) \qquad \text{for all } \delta>0 .
  \end{align*}

  First of all, the collection of laws  of the sequence
  $( (\eta_1^{\delta},   \ldots,   \eta_N^{\delta}) )_{\delta>0}$   is
  relatively         compact;         indeed,
  the RSW  estimate
  in Proposition~\ref{prop::rcm_rsw_rectangle}
  implies   the  relative  compactness  by the results
  in~\cite{AizenmanBurchardHolderRegularity, KemppainenSmirnovRandomCurves}
  (see Lemma~\ref{lem::tightness} below).
  Thus, there exist subsequential limits,  and we may assume that, for
  some    sequence $\delta_n \overset{n \to \infty}{\longrightarrow} 0$, 
  the         sequence  $(\eta_1^{\delta_n},\ldots, \eta_N^{\delta_n})$  converges weakly to
  $(\eta_1, \ldots, \eta_N)$.  For convenience, we couple  them in the
  same probability space so that they converge almost surely. Also, to
  lighten notation, we replace  the superscripts $\delta_n$ by the
  superscript  $n$  here  and  in  what  follows.  For  each
  $j\in\{1,  \ldots,  N\}$,  we  let $D_j^{n}$  denote  the  connected
  component of  $\Omega^{n}  \setminus   \bigcup_{i  \neq  j}  \eta_i^{n}$
  having  $x_{a_j}^{n}$ and  $x_{b_j}^{n}$ on its  boundary (see Figure~\ref{fig::pinching}(left)).

  In Lemma~\ref{lem::cvg_domains_Dobrushin}, we show that as $n\to\infty$, the discrete
  Dobrushin  domains  $(D_j^{n}; x_{a_j}^{n},  x_{b_j}^{n})$  converge
  almost surely to some random  Dobrushin domains in  the Carath\'{e}odory
  sense. Notice  that it is not  clear \emph{a priori} that the limit  of $D_j^{n}$ is
  still simply connected, as the interfaces in the limit may touch the
  boundary, and they  may have multiple points. The task in the
  proof of Lemma~\ref{lem::cvg_domains_Dobrushin} is therefore to rule
  out this behavior.
  We establish this by using the RSW  estimate from Proposition~\ref{prop::rcm_rsw_rectangle}.
  Specifically,  we show that  the limit
  domain $(D_j;  x_{a_j}, x_{b_j})$ is the  simply connected subdomain
  $D_j$     of     $\Omega\setminus\bigcup_{i\neq    j}\eta_i$     with
  $x_{a_j}$ and $x_{b_j}$          on         its         boundary. As a by-product,
  Lemma~\ref{lem::cvg_domains_Dobrushin}      also     shows      that
  $(\eta_1,  \ldots,  \eta_N)  \in X_0^{\alpha}(\Omega;  x_1,  \ldots,
    x_{2N})$ almost surely.

  Then, in Lemma~\ref{lem::cvg_polygon_global_multiple_subtle}  we prove
  that the subsequential limit $(\eta_1, \ldots, \eta_N)$ must be a global multiple $\SLE_{16/3}$.
  This also shows that global multiple $\SLE_{16/3}$ exists.
  Finally, in 
  Proposition~\ref{prop::globalexistence_beyond4}
  in Section~\ref{subsec::globaluniquebeyond4},
  we prove that such a global multiple $\SLE_{16/3}$ is unique, thus being
  the  unique subsequential  limit. This  gives the  convergence of  the sequence and concludes the proof of Proposition~\ref{prop::fkising_alternating_cvg}.
\end{proof}

Next, we prove the auxiliary results needed to finish the proof of Proposition~\ref{prop::fkising_alternating_cvg}.
We formulate some of them for general $q\in [1,4)$ (i.e., $\kappa\in (4,6]$), given Conjecture~\ref{conj::fkperco_cvg},
and discuss in Remark~\ref{rem::summary} the scope of these results.

\begin{lemma} \label{lem::tightness}
  Suppose Conjecture~\ref{conj::fkperco_cvg} holds for some $q\in [1,4)$ and let $\kappa\in (4,6]$ be the value related to $q$ via~\eqref{eqn::relation_q_kappa}.
  Then, the collection of laws of the sequence $( (\eta_1^{\delta},   \ldots,   \eta_N^{\delta}) )_{\delta>0}$ is relatively compact.
\end{lemma}
\begin{proof}
  The RSW estimate in Proposition~\ref{prop::rcm_rsw_rectangle} shows that the single FK-Ising interface with Dobrushin boundary conditions
  satisfies the so-called condition ``G2'' in \cite{KemppainenSmirnovRandomCurves}, and thus,
  its law is relatively compact, as stated in~\cite[Theorem~6]{DuminilSidoraviciusTassionContinuityPhaseTransition}.
  This can be generalized to conclude that also the sequence
  $( (\eta_1^{\delta},   \ldots,   \eta_N^{\delta}) )_{\delta>0}$
  of multiple interfaces with alternating boundary conditions
  is relatively compact --- see~\cite[Theorem~4.1]{Karrila19} for details.
\end{proof}

\begin{figure}[h]
  \includegraphics[width=\textwidth]{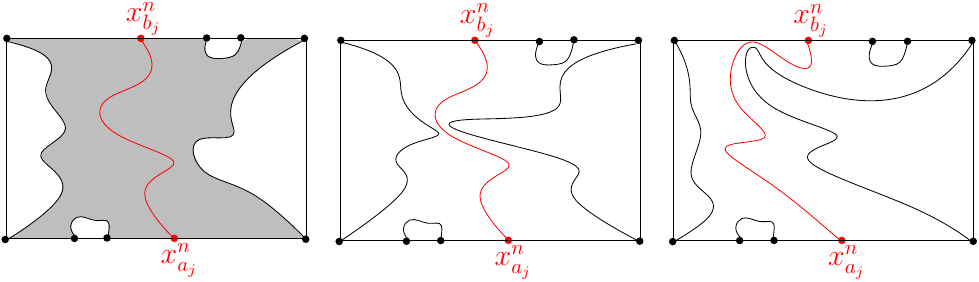}
  {\caption{\label{fig::pinching} In the left panel, the red curve is the
      interface connecting $x_{a_j}$ and $x_{b_j}$, and the gray part is
      $D_j^n$. The middle panel depicts a \textit{bulk pinching scenario}, where
      around the bulk pinching point, an interior six-arm event
      (with alternating pattern) occurs. The right panel depicts a
      \textit{boundary pinching scenario}, where around the boundary pinching
      point, a near-boundary three-arm event (with alternating
      pattern) occurs. We show in the proof of
      Lemma~\ref{lem::cvg_domains_Dobrushin} that  these events will not survive
      in the scaling limit. It turns out that in our case, the RSW
      Proposition~\ref{prop::rcm_rsw_rectangle} is sufficient for this purpose;
      note, however, that usually such bulk pinching events are ruled out by a
      six-arm exponent argument.}}
\end{figure}

\begin{lemma}\label{lem::cvg_domains_Dobrushin}
  Suppose Conjecture~\ref{conj::fkperco_cvg} holds for some $q\in [1,4)$ and let $\kappa\in (4,6]$ be the value related to $q$ via~\eqref{eqn::relation_q_kappa}.
  As  $n \to  \infty$, for
  each  $j\in\{1,   \ldots,  N\}$,   the  discrete   Dobrushin  domain
  $(D_j^{n}; x_{a_j}^{n}, x_{b_j}^{n})$ converges almost surely to the
  Dobrushin domain  $(D_j; x_{a_j}, x_{b_j})$ in  the Carath\'{e}odory
  sense.
\end{lemma}

\begin{proof}
  Fix $j\in\{1,   \ldots,  N\}$.
  As $n \to \infty$, the domains $(D_j^{n}; x_{a_j}^{n},
    x_{b_j}^{n})$ can fail to converge to a Dobrushin domain only if the limit
  domain $D_j$ is not simply connected. There are two scenarios when this
  could happen, both resulting from specific behavior of the other
  interfaces $\eta_i^{n}$ with $i \neq j$:
  either two of these interfaces get close together in the interior of $\Omega^n$, pinching $\eta_j^{n}$ in between
  (see   Figure~\ref{fig::pinching} (middle)), or
  one of these interfaces gets close to the  boundary of $\Omega^n$, pinching $\eta_j^{n}$ to the boundary
  (see Figure~\ref{fig::pinching}(right)).
  In both cases, 
  the points $x_{a_j}^{n}$ and $x_{b_j}^{n}$ get disconnected in
  the limit $n \to \infty$. We call the former a \textit{bulk pinching
    scenario} and the latter a \textit{boundary pinching scenario}.

  \begin{figure}[h]
    \includegraphics[width=\textwidth]{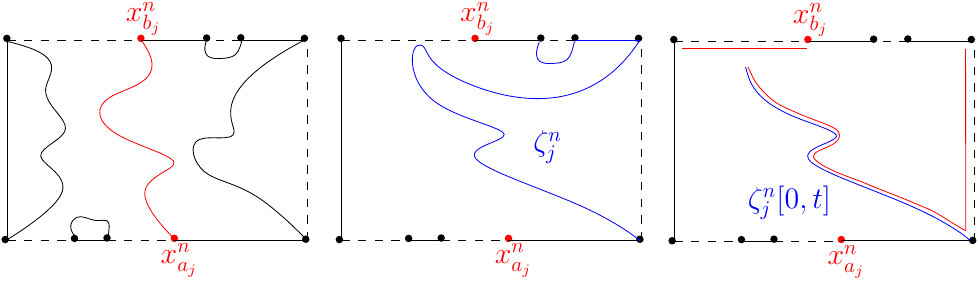}
    \caption{\label{fig::boundarypinching}
      A typical boundary pinching scenario.
    }
  \end{figure}

  First, we consider the boundary pinching scenario. Without loss of generality,
  we may assume that the boundary conditions on $(x^n_{a_j} \, x^n_{a_j+1})$
  are wired, as in Figure~\ref{fig::boundarypinching} (left). Also, it
  suffices to consider the pinching on the boundary arc $(x^n_{b_j} \,
    x^n_{b_j+1})$ and assume that $b_j\ge a_j+2$. Denote by $\LC_j^n$ the event
  that there is an open path connecting $(x^n_{a_j} \, x^n_{a_j+1})$ to
  $(x^n_{b_j-1} \, x^n_{b_j})$ in $\Omega^{n}$. Note that $\{\conn^n=\alpha\}$
  implies the event $\LC_j^n$. Denote the exploration path from $x^n_{a_j+1}$
  to $x^n_{b_j-1}$ by $\path_j^{n}$, as in Figure~\ref{fig::boundarypinching}
  (middle), parameterized  by the number  of steps starting from $x^n_{a_j+1}$.
  For a fixed time $t$, inside the domain
  $\Omega^n \setminus\path_j^{n}[0,t]$, consider the two boundary arcs
  $\partial_1^n := (x^n_{b_j} \, x^n_{b_j+1})$ and $\partial_2^n$ defined as the
  union of the boundary arc $(x^n_{a_j+1} \, x^n_{a_j+2})$ and the right side
  of $\path_j^{n}[0,t]$ --- both carry free boundary conditions, and are drawn
  in red on Figure~\ref{fig::boundarypinching}(right).
  Notice that the event $\LC_j^n$
  implies that there is no dual-open crosscut between $\partial_1^n$ and
  $\partial_2^n$. The gist of the argument is that if a boundary pinching
  occurs, such a crosscut will exist with high probability.

  For all $t \geq 0$, let $d_1^n(t)$ denote the length of the shortest path in
  $D_j^n$ between $\path_j^n(t)$ and $\partial_1^n$ that does not intersect
  $\path_j^n[0,t]$, and set $d_2^n(t) := | x_{b_j}^n - \path_j^n(t) |$   and $\varepsilon_j^n(t) := d_1^n(t)/d_2^n(t)$.
  Then, the RSW estimate from Proposition~\ref{prop::rcm_rsw_rectangle}
  combined with the FKG inequality~\eqref{eqn::rcm_boundary_comparison}
  shows that for some universal constant $C>0$, we have the upper bound
  $\PP[\LC_j^n \; | \; \path_j^n[0,t]] \leq C (\varepsilon_j^n(t))^{1/C}$. Now, for
  $u>0$ small, denote by $T_u$ the first time $t \geq 0$ when $\varepsilon_j^n(t) \leq u$
  (equaling $+\infty$ if no such time exists). From the above bound, we obtain
  \begin{align*}
    \PP \big[ \underset{t \geq 0}{\inf} \, \varepsilon_j^n(t) \leq u \; \big| \; \vartheta^n = \alpha \big]
    \leq \; & \frac{\PP \big[ \LC_j^n \cap \big\{ \underset{t \geq 0}{\inf} \, \varepsilon_j^n(t) \leq u \big\} \big] }{\PP [\vartheta^n = \alpha]} \\
    = \;    &
    \frac{\E \big[ \one_{\{T_u<\infty\}} \, \E[\LC_j^n \; | \; \path_j^n[0,T_u]] \big]} {\PP[\vartheta^n =
      \alpha]}
    \; \leq \; \frac {C u^{1/C}} {\PP[\vartheta^n = \alpha]}.
  \end{align*}
  Proposition~\ref{prop::rcm_rsw_rectangle} (cf. the footnote in the proof of
  Lemma~\ref{lem::positivechance_coincide_general_beyond4}) implies that
  $\PP[\conn^n=\alpha]$ is bounded away from zero uniformly in $n$.
  Therefore,  we have
  \begin{align*}
    \lim_{u\to 0}\limsup_{n\to \infty}\PP \big[ \underset{t \geq 0}{\inf} \, \varepsilon_j^n(t) \le u \; \big| \; \conn^n=\alpha \big] = 0 .
  \end{align*}
  This shows that, in the scaling limit $n\to\infty$, the boundary pinching scenario cannot occur.

  Bulk pinchings can be ruled out as a consequence. Indeed, assume that 
  on the boundary of the domain $\Omega^{n}$, there is a triple of pairs of boundary points
  belonging to the pairing $\vartheta^{n}$ and such that the corresponding
  interfaces, say $\path_1^{n}$, $\path_2^{n}$, and $\path_3^{n}$, are involved in a bulk
  pinching scenario with positive probability (see Figure~\ref{fig::pinching}(middle), where
  $\path_2^{n}$ is colored red). First, explore $\path_1^{n}$; such a bulk pinching can
  then be seen as a boundary pinching in the complement of $\path_1^{n}$, and such
  boundary pinchings are excluded by the previous argument.

  In summary, we have shown that  neither the bulk pinching
  scenario nor the boundary pinching scenario can survive in the scaling limit.
  This shows that  $(D_j^n; x_{a_j}^n, x_{b_j}^n)$ converges almost surely to
  the Dobrushin domain $(D_j; x_{a_j}, x_{b_j})$ in the Carathéodory sense,
  which is what we sought to prove.
\end{proof}

Note that the proof of Lemma~\ref{lem::cvg_domains_Dobrushin} also shows that
almost surely,
\[(\eta_1,   \ldots,  \eta_N)\in   X_0^{\alpha}(\Omega;  x_1,   \ldots, x_{2N}).\]

\begin{lemma}\label{lem::cvg_polygon_global_multiple_subtle}
  In      the      setup     of      the      proof      of
  Proposition~\ref{prop::fkising_alternating_cvg} \textnormal{(}with $q=2$ and $\kappa = 16/3$\textnormal{)},         the        limit
  $(\eta_1, \ldots, \eta_N)$ has the distribution of a global multiple
  $\SLE_{16/3}$.
\end{lemma}

\begin{proof}
  We need to prove that, for each $j\in\{1, \ldots, N\}$, the conditional law of the random curve $X:=\eta_j$
  given the other random curves $Y:=(\eta_1, \ldots, \eta_{j-1}, \eta_{j+1}, \ldots, \eta_N)$ is the appropriate chordal $\SLE_{16/3}$.
  We fix $j$ and denote
  \begin{align*}
    X^n := \eta_j^n \qquad \text{and} \qquad Y^n := (\eta_1^n, \ldots, \eta_{j-1}^n, \eta_{j+1}^n, \ldots, \eta_N^n) , \qquad n \geq 1 .
  \end{align*}

  By assumption, $(X^n, Y^n)$ converges to $(X, Y)$ in distribution. 
  However, this does not automatically imply the convergence of the conditional distribution of $X^n$ given $Y^n$ to the conditional distribution of $X$ given~$Y$.
  In our case this is true, as we will now prove.
  (See also the discussion in~\cite[Section~5]{GarbanWuDustAnalysisFKIsing}.)

  Recall that we couple all of the random variables 
  $\{ (X^n, Y^n) \colon n \geq 1\}$
  in the same probability space so that they converge almost surely to $(X, Y)$ as 
  $n \to \infty$.
  Now, given  $Y^n$, the random curve  $X^n$ is a FK-Ising interface with Dobrushin boundary conditions in
  the random Dobrushin domain $(D_j^n; x^n_{a_j}, x^n_{b_j})$ by the domain Markov property.
  By  Lemma~\ref{lem::cvg_domains_Dobrushin}, $(D_j^n; x^n_{a_j}, x^n_{b_j})$ converges almost surely to
  the random Dobrushin domain $(D_j; x_{a_j}, x_{b_j})$  in the Carath\'{e}odory sense.
  Thus, almost surely, there exist conformal maps $G^n$ (resp.~$G$) from $\U$ onto $D_j^n$ (resp.~$D_j$) such that, as 
  $n \to \infty$,
  the maps $G^n$ converge to $G$ uniformly on compact subsets of $\U$, and we have
  $(G^n)^{-1}(x_{a_j}^{n}) \to G^{-1}(x_{a_j}) = 1$ and $(G^n)^{-1}(x_{b_j}^{n}) \to G^{-1}(x_{b_j}) = -1$.
  Furthermore, for each $n$, the map $G^n$ is a measurable function of $Y^n$, and $G$ is a measurable function of~$Y$.
  To conclude, we use the following two observations.
  \begin{enumerate}

    \item \label{Item1} On the one hand, Theorem~\ref{thm::fkising_cvg}  shows  that the  law  of
          $(G^n)^{-1}(X^n)$ converges  to the chordal  $\SLE_3$ in $\U$  connecting the points $1$  and $-1$.

    \item \label{Item2}
          On the other hand (see also~\cite[Proposition~4.7]{Karrila}), one can show that
          $(G^n)^{-1}(X^n)$ converges  to $G^{-1}(X)$ as follows.
          By assumption, $(X^n, Y^n)$ converges to  $(X, Y)$  almost  surely. Now, we
          send $X^n$ (resp.~$X$) conformally onto $\HH$ and denote by $W^n$ (resp.~$W$) its driving function.
          On the one hand, applying
          \cite[Proposition~4.12, Theorem~1.5, and Corollary~1.7]{KemppainenSmirnovRandomCurves}
          to the critical FK-Ising interfaces  $(X^n)_{n \geq 1}$, we see that $W^n\to W$ locally uniformly.
          On the other hand, applying
          \cite[Proposition~4.12, Theorem~1.5, and Corollary~1.7]{KemppainenSmirnovRandomCurves}
          to  $\{ (G^n)^{-1}(X^n) \}$,
          we see that this collection is tight, and for any convergent subsequence $(G^{n_k})^{-1}(X^{n_k})\to \tilde{\eta}$,
          the curve $\tilde{\eta}$ has a continuous driving function $\widetilde{W}$ such that
          $W^{n_k} \to \widetilde{W}$ locally uniformly (note that this fact is highly non-trivial).
          Combining these two facts, we see that $\widetilde{W}$ coincides with $W$, so $\tilde{\eta}$ coincides with $G^{-1}(X)$.
          In particular, this is the only subsequential limit of the collection 
          $\{ (G^n)^{-1}(X^n) \colon n \geq 1 \}$,
          so we have $(G^n)^{-1}(X^n)\to G^{-1}(X)$ as $n \to \infty$.
  \end{enumerate}

  Combining these observations,
  we see  that  the law  of  $G^{-1}(X)$  is the  chordal
  $\SLE_{16/3}$ in $\U$  connecting $1$ and $-1$. 
  In  particular, the law of $G^{-1}(X)$ is independent of $Y$ with $G$ a measurable function of $Y$.
  Hence, the conditional law  of $X$ given $Y$ is the chordal
  $\SLE_{16/3}$ in $D_j$ connecting the points $x_{a_j}$ and~$x_{b_j}$.
\end{proof}

\subsection{Uniqueness of Global Multiple $\SLE$s with $\kappa\in (4,6]$}
\label{subsec::globaluniquebeyond4}

In this section, we prove that the scaling limit of each subsequence of FK-Ising interfaces is unique, thereby finishing the proof of Proposition~\ref{prop::fkising_alternating_cvg}.
The idea is similar to the proof of Theorem~\ref{thm::global_unique} in Section~\ref{subsec::uniqueness_general}.
In particular, we need analogues of the lemmas appearing in Sections~\ref{subsec::uniqueness_pair} and~\ref{subsec::uniqueness_general}.
Again, we formulate them for general $\kappa\in (4,6]$.

\begin{lemma}\label{lem::sle_positivechance_stay_beyond4}
  Suppose Conjecture~\ref{conj::fkperco_cvg} holds for some $q\in [1,4)$ and let $\kappa\in (4,6]$ be the value related to $q$ via~\eqref{eqn::relation_q_kappa}.
  Let $(\Omega; x, y)$ be a bounded Dobrushin domain. Let $\Omega^L, U \subset \Omega$ be Dobrushin subdomains
  such that $\Omega^L$, $U$, and $\Omega$ agree in a neighborhood of the arc $(y \, x)$.
  Suppose $\gamma \sim \PP(\Omega; x, y)$ and $\eta \sim \PP(U; x, y)$.
  Then, we have
  \begin{align*}
  \PP[\eta\subset\Omega^L]\ge\PP[\gamma\subset\Omega^L].
  \end{align*}
  In particular,  Lemma~\ref{lem::sle_positivechance_stay} holds for the corresponding $\kappa\in (4,6]$.
\end{lemma}
\begin{proof}
  This immediately follows by combining the domain Markov property with
  the comparison~\eqref{eqn::rcm_boundary_comparison} of boundary conditions
  with Conjecture~\ref{conj::fkperco_cvg}.
\end{proof}

We remark that Lemma~\ref{lem::sle_positivechance_stay_beyond4} concerns the chordal $\SLE_{\kappa}$ with $\kappa\in (4,6]$,
and its statement has nothing to do with discrete models.
However, we do not have 
a proof for this lemma without using Conjecture~\ref{conj::fkperco_cvg}.

\begin{proposition}\label{prop::two_sle_beyond4}
  Suppose Conjecture~\ref{conj::fkperco_cvg} holds for some $q\in [1,4)$ and let $\kappa\in (4,6]$ be the value related to $q$ via~\eqref{eqn::relation_q_kappa}.
  Then, for each quad $(\Omega; x_1, \ldots,  x_{4})$ and for each link pattern $\alpha \in \LP_2$, there exists a unique global $2$-$\SLE_{\kappa}$ associated to~$\alpha$.
\end{proposition}
\begin{proof}
  As in Section~\ref{subsec::uniqueness_pair}, without loss of generality, we assume that $\alpha = \{\link{1}{4}, \link{2}{3}\}$.
  Then, to prove the assertion, we argue as in the proof of Proposition~\ref{prop::slepair_unique},
  with $(\Omega; x_1, \ldots,  x_{4}) = (\Omega; x^L, x^R, y^R, y^L)$. Taking
  $\Omega=[0,\ell]\times[0,1]$ and $x^L=(0,0)$, $x^R=(\ell, 0)$, $y^R=(\ell, 1)$, $y^L=(0,1)$, we
  define a Markov chain on pairs $(\eta^L,\eta^R)$ of curves by sampling from the conditional laws:
  given $(\eta^L_n, \eta^R_n)$, we pick $i\in\{L, R\}$ uniformly and resample $\eta^i_{n+1}$ according to the conditional law given the other curve.
  However, in the current situation, we have $\kappa\in (4,6]$, so the configuration sampled according to this rule may no longer stay in the space
  $X_0(\Omega; x^L, y^L, x^R, y^L)$.
  In this case, when resampling according to the conditional law, we sample the curves in each connected component and concatenate the pieces of curves together
  --- see the more detailed description beneath Equation~\eqref{eq::dim_of_A1}.
  Fortunately, this issue turns out to be irrelevant in the end, as we will show that,
  for any initial configuration $(\eta^L_0, \eta^R_0)\in X_0(\Omega; x^L, x^R, y^R, y^L)$,
  the corresponding Markov chain $(\eta^L_n, \eta^R_n)$ will eventually stay in the space $X_0(\Omega; x^L, y^L, x^R, y^L)$:
  \begin{align}\label{eqn::slepair_uniqueness_beyond4_aux}
    \PP\left[\exists \; n_0 < \infty \text{ such that }(\eta^L_n, \eta^R_n)\in X_0(\Omega; x^L, y^L, x^R, y^L) \text{ for all } n\ge n_0 \right]=1.
  \end{align}
  Once~\eqref{eqn::slepair_uniqueness_beyond4_aux} has been established, the existence and uniqueness of the global $2$-$\SLE_{\kappa}$ follows by
  repeating the proof Proposition~\ref{prop::slepair_unique}, with Lemma~\ref{lem::sle_positivechance_stay} replaced by Lemma~\ref{lem::sle_positivechance_stay_beyond4}.
  Hence, it remains to prove~\eqref{eqn::slepair_uniqueness_beyond4_aux}.

  In the Markov chain $(\eta^L_n, \eta^R_n)$, we want to record the times when $L$ and $R$ are picked. Let $\tau_0^L=\tau_0^R=0$,
  and for $n \geq 1$, let $\tau_n^R$ (resp.~$\tau_n^L$) be the first time after $\tau_{n-1}^L$ (resp.~$\tau_n^R$) that $R$ (resp.~$L$) is picked.
  Let
  \begin{align*}
    n_{\kappa} = \bigg\lceil \frac{\kappa}{8-\kappa} \bigg\rceil +1.
  \end{align*}
  To prove~\eqref{eqn::slepair_uniqueness_beyond4_aux}, it suffices to show that
  $\eta^R_n\cap (y^L \, x^L)=\emptyset$ for all $n\ge \tau_{n_{\kappa}}^R$, because a similar property for $\eta^L_n$ follows by symmetry
  (note also that $\tau_{n}^L\ge \tau_n^R$).
  For this purpose, we let $\gamma^R$ be the $\SLE_{\kappa}$ in $\Omega$ connecting $x^R$ and $y^R$.
  We will use the following two essential properties of $\gamma^R$:
  \begin{enumerate}
    \item \label{item1} By the duality property of the $\SLE_{\kappa}$ (see, e.g.,~\cite{DubedatSLEDuality} or~\cite[Theorem 1.4]{MillerSheffieldIG1}),
          we know that the left boundary of $\gamma^R$ has the law of the $\SLE_{16/\kappa}(16/\kappa-4; 8/\kappa-2)$
          with two force points next to the starting point. Therefore, the  left boundary of $\gamma^R$ does not hit $(x^Ry^R)$.

    \item \label{item2} The curve  $\gamma^R$ hits $(y^L \, x^L)$ with positive probability, and
          using~\cite{AlbertsKozdronIntersectionProbaSLEBoundary} and Lemma~\ref{lem::appendix} from appendix~\ref{sec::appendix_fractal},
          we see that, almost surely on the event $\{\gamma^R\cap(y^L \, x^L)\neq\emptyset\}$,
          the Hausdorff dimension of the intersection set satisfies
          \begin{align*}
            \dimH(\gamma^R\cap (y^L \, x^L))\le 1-\beta, \qquad \text{where }      \quad       \beta = \frac{8-\kappa}{\kappa} .
          \end{align*}
  \end{enumerate}
  Now, for $\tau_1^R\le n\le \tau_1^L-1$, the curve $\eta^R_n$ is an $\SLE_{\kappa}$ in a domain which is a subset of $\Omega$. By Lemma~\ref{lem::sle_positivechance_stay_beyond4}, we can couple $\eta^R_n$ and $\gamma^R$ 
  so that $\gamma^R$ stays to the left of $\eta^R_n$ almost surely. Thus, we have almost surely
  \begin{align*}
    \dimH(\eta^R_n\cap (y^L \, x^L))\le \dimH(\gamma^R\cap (y^L \, x^L))\le 1-\beta.
  \end{align*}
  In particular, for the last time before sampling the left curve, we have almost surely
  \begin{align}\label{eq::dim_of_A1}
    \dimH(A_1)\le 1-\beta \qquad \text{for} \qquad A_1=\eta^R_{\tau_1^L-1}\cap (y^L \, x^L).
  \end{align}

  Then, for $\tau^L_1\le n\le \tau_2^R-1$, we sample $\eta^L_n$ according to the conditional law given $\smash{\eta^R_{\tau^L_1-1}}$.
  However,  if $A_1\neq\emptyset$, then the domain $\smash{\Omega\setminus \eta^R_{\tau_1^L-1}}$ is not 
  connected.
  In this case, we sample the $\SLE_{\kappa}$ in those connected components of $\smash{\Omega\setminus \eta^R_{\tau_1^L-1}}$
  which have a part of $(y^L \, x^L)$ on the boundary and define $\eta^L_n$ to be the concatenation of these curves.
  We note that, by the above observation~\ref{item1}, the right boundary of $\eta^L_n$ only hits $(y^L \, x^L)$ in $A_1$.

  Next, for $\tau_2^R\le n\le \tau_2^L-1$, we sample $\eta_n^R$ according to the conditional law given $\smash{\eta^L_{\tau_2^R-1}}$.
  Again, the curve $\eta_n^R$ is an $\SLE_{\kappa}$ in a domain which is a subset of $\Omega$, and 
  we can couple 
  it with $\gamma^R$ in such a way that $\gamma^R$ stays to the left of $\eta^R_n$ almost surely.
  Thus, we have almost surely
  \begin{align*}
    \eta^R_n\cap (y^L \, x^L) \; \subset \;  \eta^R_n\cap A_1 \;  \subset \;  \gamma^R\cap A_1.
  \end{align*}
  Combining this with~\eqref{eq::dim_of_A1}, we see that, almost surely,
  \begin{align*}
    \dimH(\eta^R_n\cap (y^L \, x^L))\le \dimH(\gamma^R\cap A_1)\le (1-2\beta)^+.
  \end{align*}
  In particular, we can improve~\eqref{eq::dim_of_A1} to
  \begin{align*} 
    \dimH(A_2)\le (1-2\beta)^+ \qquad \text{for} \qquad A_2=\eta^R_{\tau_2^L-1}\cap (y^L \, x^L),
  \end{align*}
and iterating the same argument and combining with Lemma~\ref{lem::appendix}, we see that 
  \begin{align*}
    \eta_n^R\cap (y^L \, x^L)=\emptyset \qquad \text{for all } n\ge \tau_{n_{\kappa}}^R ,
  \end{align*}
  almost surely.  This concludes the proof.
\end{proof}

\begin{proposition}\label{prop::globalexistence_beyond4}
  Let $(\Omega; x_1, \ldots,  x_{2N})$ be a
  polygon with $N \geq 1$. For  any $\alpha \in \LP_N$, there exists a
  unique global $N$-$\SLE_{16/3}$ associated to~$\alpha$.
\end{proposition}

\begin{proof}
  The existence follows from the subsequential scaling limit in Lemma~\ref{lem::cvg_polygon_global_multiple_subtle}, so it remains to prove the uniqueness.
  We use induction on $N \geq 2$ and the same arguments as in the proof of Theorem~\ref{thm::global_unique}.
  First, the assertion holds for $N=2$ by 
  Proposition~\ref{prop::two_sle_beyond4}.
  Next, we let $N\ge 3$ and assume that for any $\beta\in\LP_{N-1}$, the global $(N-1)$-$\SLE_{\kappa}$ associated to $\beta$ is unique.
  Then, as in the proof of Theorem~\ref{thm::global_unique}, we take $\alpha\in\LP_N$ with
  $\link{1}{2}\in\alpha$ and $\link{r}{r+1}\in\alpha$ for some $r\in\{3, 4, \ldots, 2N-1\}$, and we
  let $(\eta_1, \ldots,\eta_N)\in X_0^{\alpha}(\Omega; x_1, \ldots, x_{2N})$ be a global $N$-$\SLE_{\kappa}$ associated to $\alpha$.
  We denote by $\eta^L$ (resp.~$\eta^R$) the curve in the collection $\{\eta_1, \ldots, \eta_N\}$ that connects $x_1$ and $x_2$ (resp.~$x_{r}$ and $x_{r+1}$).
  By the induction hypothesis, given $\eta^R$ (resp.~$\eta^L$), the conditional law of the rest of the curves is the unique global $(N-1)$-$\SLE_{\kappa}$
  associated to $\alpha\removeLink \link{r}{r+1}$ (resp.~$\alpha\removeLink \link{1}{2}$).
  This gives the conditional law of $\eta^L$ given $\eta^R$ and 
  vice versa.
  One can then use the argument from the proof of Proposition~\ref{prop::slepair_unique}, considering Markov chains sampling $\eta^L$ and $\eta^R$
  from their conditional laws --- one only has to replace Lemma~\ref{lem::sle_positivechance_stay} by Lemma~\ref{lem::sle_positivechance_stay_beyond4}
  and Lemma~\ref{lem::sle_positivechance_coincide} by the following Lemma~\ref{lem::positivechance_coincide_general_beyond4} for $N-1$.
\end{proof}

The next technical lemma can be thought of as an analogue of Lemma~\ref{lem::sle_positivechance_coincide}.
To state it, we fix $\alpha\in\LP_N$ such that $\link{1}{2}\in\alpha$ and let $(\Omega; x_1, \ldots, x_{2N})$ be a bounded polygon.
Also, if $(\eta_1, \ldots, \eta_N)$ is a family of random curves with the law of a global $N$-$\SLE_{\kappa}$ associated to $\alpha$,
and if $\eta := \eta_1$ is the curve connecting $x_1$ and $x_2$, then we denote by $\QQ_{\alpha}^{\link{1}{2}}(\Omega; x_1, \ldots, x_{2N})$ the law of~$\eta$.

\begin{lemma}\label{lem::positivechance_coincide_general_beyond4}
  Assume that there exists a unique global $N$-$\SLE_{16/3}$ associated to
  $\alpha$. Let $\Omega^L \subset U, \tilde{U} \subset \Omega$ be sub-polygons such that
  $\Omega^L$ and $\Omega$ agree in a neighborhood of the boundary arc $(x_1 \, x_2)$.
  Suppose that $\eta \sim \QQ_{\alpha}^{\link{1}{2}}(U; x_1, \ldots, x_{2N})$ and $\tilde{\eta} \sim \QQ_{\alpha}^{\link{1}{2}}(\tilde{U}; x_1, \ldots, x_{2N})$.
  Then, there exists a coupling $(\eta, \tilde{\eta})$ such that $\PP[\eta=\tilde{\eta}\subset\Omega^L]\ge\theta$,
  where the constant $\theta=\theta(\Omega, \Omega^L)>0$ is independent of $U$ and $\tilde{U}$.
\end{lemma}
\begin{proof}
  Let $(\Omega^{\delta}; x_1^{\delta}, \ldots, x_{2N}^{\delta})$
  be discrete polygons converging to $(\Omega; x_1, \ldots, x_{2N})$ 
  in the Carath\'eodory sense,
  and denote by $U^{\delta}$, $\tilde{U}^{\delta}$, and $(\Omega^L)^{\delta}$ the corresponding approximations of $U$, $\tilde{U}$, and $\Omega^L$.
  Also, let $(\eta_1^{\delta}, \ldots, \eta_N^{\delta})$ (resp.~$(\tilde{\eta}_1^{\delta}, \ldots, \tilde{\eta}_N^{\delta})$)
  be the collection of interfaces in the critical random-cluster model on $U^{\delta}$ (resp.~$\tilde{U}^{\delta}$) with alternating boundary
  conditions~\eqref{eq::FK_alternating}, and let $\eta^{\delta} := \eta_1^{\delta}$ and $\tilde{\eta}^{\delta} := \tilde{\eta}^{\delta}_1$
  be the curves connecting $x_1^{\delta}$ and $x_2^{\delta}$.
  By the assumptions, we know that, as $\delta \to 0$,
  the law of $\eta^{\delta}$ (resp.~$\tilde{\eta}^{\delta}$) conditionally on $\{\conn^{\delta}=\alpha\}$
  (resp.~$\{\tilde{\conn}^{\delta}=\alpha\}$) converges to $\QQ_{\alpha}^{\link{1}{2}}(U; x_1, \ldots, x_{2N})$ (resp.~$\QQ_{\alpha}^{\link{1}{2}}(\tilde{U}; x_1, \ldots, x_{2N})$).
  Thus, it is sufficient to show
  that there exists a coupling $(\eta^{\delta}, \tilde{\eta}^{\delta})$ such that $\PP[\eta^{\delta} = \tilde{\eta}^{\delta} \subset (\Omega^L)^{\delta}] \ge \theta$ for $\delta$ small enough, where the constant $\theta=\theta(\Omega, \Omega^L)>0$ is independent of $U$ and $\tilde{U}$.

  Since $\Omega^L$ agrees with $\Omega$ in a neighborhood of $(x_1 \, x_2)$, we can find
  boundary points $y_1$ and $y_2$ such that $y_1, x_1, x_2, y_2$ lie
  in counterclockwise order along $\partial \Omega$ and $\Omega^L$ agrees with $\Omega$ in a neighborhood of $(y_1 \, y_2)$.
  Now, we have wired boundary conditions on the arc $(x_1^{\delta} \, x_2^{\delta})$
  and free boundary conditions on the arcs $(x_2^{\delta} \, x_3^{\delta})$ and $(x_{2N}^{\delta} \, x_1^{\delta})$.
  Define $\LC_*^{\delta}$
  to be the event that there exists a dual-open path in $(\Omega^L)^{\delta}$ from $(x_2^{\delta} \, y_2^{\delta})$ to $(y_1^{\delta} \, x_1^{\delta})$.
  Then, by the domain Markov property,
  there exists a coupling of $\eta^{\delta}$ and $\tilde{\eta}^{\delta}$ such that the probability of $\{\eta^{\delta}=\tilde{\eta}^{\delta}\subset(\Omega^L)^{\delta}\}$
  is bounded from below by the minimum of $\PP[\LC^{\delta}_*]$ and $\tilde{\PP}[\LC^{\delta}_*]$,
  where $\PP$ and $\tilde{\PP}$ denote the probability measures of the random-cluster models on $U^{\delta}$ and $\tilde{U}^{\delta}$
  with alternating boundary conditions~\eqref{eq::FK_alternating}.
  Furthermore, as a consequence of
  Proposition~\ref{prop::rcm_rsw_rectangle}, the domain Markov property,
  and the FKG inequality~\eqref{eqn::rcm_boundary_comparison}, we obtain
  $\PP[\LC^{\delta}_*] \ge \theta(\Omega, \Omega^L) > 0$
  (and likewise for $\tilde{U}$)\footnote{Note that, here we only need the RSW Proposition~\ref{prop::rcm_rsw_rectangle},
  because the lower bound $\theta$ is allowed to depend on the domains $\Omega, \Omega^L$.
  To derive $\theta(\Omega, \Omega^L)$ from Proposition~\ref{prop::rcm_rsw_rectangle},
  one can draw a zigzag path of rectangles so that the first one intersects the boundary arc $(y_1^{\delta} \, x_1^{\delta})$,
  the last one intersects the boundary arc $(x_2^{\delta} \, y_2^{\delta})$, and the middle ones are inside $(\Omega^L)^{\delta}$,
  and observe that dual crossings of all these rectangles give a dual crossing from $(y_1^{\delta} \, x_1^{\delta})$ to $(x_2^{\delta} \, y_2^{\delta})$.}.
  In particular, the lower bound $\theta(\Omega, \Omega^L)$ is uniform over $U$ (resp.~$\tilde{U}$) and $\delta$.
  By the convergence of $\eta^{\delta}$ and $\tilde{\eta}^{\delta}$, we obtain a coupling of $\eta$ and $\tilde{\eta}$ such that the probability
  of $\{\eta = \tilde{\eta} \subset \Omega^L\}$ is bounded from below by $\theta(\Omega,\Omega^L)$. 
\end{proof}

By the above, we have now completed the proof of Proposition~\ref{prop::fkising_alternating_cvg} (with $q=2$ and $\kappa = 16/3$).
We summarize the key ingredients in the proof in the following remark.

\begin{remark} \label{rem::summary}
  The proof of Proposition~\ref{prop::fkising_alternating_cvg} consists of  Lemmas~\ref{lem::tightness}~--~\ref{lem::positivechance_coincide_general_beyond4}
  and Propositions~\ref{prop::two_sle_beyond4} and~\ref{prop::globalexistence_beyond4}.
  \begin{itemize}
    \item Lemmas~\ref{lem::tightness},~\ref{lem::cvg_domains_Dobrushin}, and~\ref{lem::cvg_polygon_global_multiple_subtle}
          require the RSW estimate from Proposition~\ref{prop::rcm_rsw_rectangle}.

    \item 
          Lemma~\ref{lem::sle_positivechance_stay_beyond4}
          requires the convergence of a single interface, given by Conjecture~\ref{conj::fkperco_cvg}.

    \item The proof of  
            {Proposition~\ref{prop::two_sle_beyond4}}
          uses Lemma~\ref{lem::sle_positivechance_stay_beyond4}. 
          Assuming Lemma~\ref{lem::sle_positivechance_stay_beyond4}, this works for all $\kappa\in (4,8)$.

    \item Note also that the proof of Proposition~\ref{prop::two_sle_beyond4} uses
          the duality of the $\SLE_\kappa$, which is known for all $\kappa\in (4,8)$~\textnormal{\cite{DubedatSLEDuality,MillerSheffieldIG1}}. 

    \item The proofs of  
          Proposition~\ref{prop::globalexistence_beyond4} and Lemma~\ref{lem::positivechance_coincide_general_beyond4}
          use the convergence of the multiple FK-Ising interfaces;
          thus they also require Lemmas~\ref{lem::tightness}~--~\ref{lem::sle_positivechance_stay_beyond4} and Proposition~\ref{prop::two_sle_beyond4} as an input.
  \end{itemize}
  Overall, the proofs of these results require the RSW estimate from Proposition~\ref{prop::rcm_rsw_rectangle} 
  and the convergence a single interface (Conjecture~\ref{conj::fkperco_cvg}).
  Therefore, knowing Conjecture~\ref{conj::fkperco_cvg}, the analogous
  conclusions to Proposition~\ref{prop::fkising_alternating_cvg} would extend
  to other values of $\kappa$.
\end{remark}

\subsection{The Ising Model}
\label{subsec::ising_multiple_cvg}

Let $\graph$ be a finite subgraph of $\Z^2$. The Ising model on $\graph$ with free boundary condition is a random assignment
$\sigma \in \{\ominus, \oplus\}^{V(\graph)}$ of spins $\sigma_v \in \{\ominus, \oplus\}$, where $\sigma_v$ denotes the spin at the vertex $v \in V(\graph)$.
The Hamiltonian is defined by
\begin{align*}
  H^{\free}_{\graph}(\sigma) = - \sum_{v \sim w}\sigma_v \sigma_w .
\end{align*}
The probability measure of the Ising model is given by the Boltzmann measure with Hamiltonian $H^{\free}_{\graph}$ and inverse-temperature $\beta>0$:
\begin{align*}
  \mu^{\free}_{\beta,\graph}[\sigma]
  = \frac{\exp(-\beta H^{\free}_{\graph}(\sigma))}{Z^{\free}_{\beta, \graph}}, \qquad
  \text{where } \quad Z^{\free}_{\beta, \graph}=\sum_{\sigma}\exp(-\beta H^{\free}_{\graph}(\sigma)) .
\end{align*}
Also, for $\tau\in \{\ominus, \oplus\}^{\Z^2}$, we define the Ising model with boundary condition $\tau$ via the Hamiltonian
\begin{align*}
  H^{\tau}_{\graph}(\sigma) = -\!\!\!\!\!\!\!\! \sum_{\substack{v \sim w, \\ \edge{v}{w} \cap \graph \neq \emptyset}} \sigma_v \sigma_w, \qquad
  \text{where } \quad \sigma_v = \tau_v  \text{ for all } v \not\in \graph .
\end{align*}
In particular, if $(\graph; v, w)$ is a discrete Dobrushin domain, we may consider the Ising model with the following \textit{Dobrushin boundary conditions}
(domain-wall boundary conditions): we set
$\oplus$ along the arc $(v \, w)$, and $\ominus$ along the complementary arc $(w \, v)$. More generally, we will consider the alternating boundary conditions~\eqref{eqn::ising_bc_alternating},
where $\oplus$ and $\ominus$ alternate along the boundary as in Figure~\ref{fig::Ising}.

\smallbreak

As in the case of the random-cluster model, we have the    following  useful   domain    Markov    property.
Let $\graph  \subset  \graph'$ be  two  finite  subgraphs of  $\Z^2$.  Fix
$\tau  \in \{\ominus,  \oplus\}^{\Z^2}$ and  $\beta>0$. Let  $X$ be  a
random variable, which is measurable with respect to the status of the
vertices in the smaller graph $\graph$. Then, we have
\begin{align*}
  \mu^{\tau}_{\beta, \graph'}  \big[X \; | \; \sigma_v = \tau_v  \text{ for all } v \in \graph' \setminus \graph \big] = \mu^{\tau}_{\beta, \graph}[X] .
\end{align*}

\smallbreak

The planar Ising model exhibits an order-disorder phase transition at a certain critical temperature:
above this temperature, the configurations are disordered and below it, one observes large clusters of equal spins.
At criticality, the configurations show self-similar behavior, and indeed, the critical planar Ising model is conformally invariant in the scaling
limit~\cite{SmirnovConformalInvariance, SmirnovConformalInvarianceAnnals, ChelkakSmirnovIsing, HonglerSmirnovIsingEnergy, ChelkakHonglerLzyurovConformalInvarianceCorrelationIsing, CDCHKSConvergenceIsingSLE}.
On the square lattice, the critical value of $\beta$ is 
\[\beta_c=\frac{1}{2}\log(1+\sqrt{2}).\]

In Proposition~\ref{prop::ising_multiple_cvg}, we consider the scaling limit of Ising interfaces at criticality.
Let $(\Omega_*^{\delta}; x_*^{\delta}, y_*^{\delta})$ be discrete Dobrushin domains, $\delta > 0$,
and consider the critical Ising model on the duals $(\Omega_*^{\delta}; x_*^{\delta}, y_*^{\delta})$ with Dobrushin boundary conditions.
Let $x^{\delta}_{\diamond}$ and $y^{\delta}_{\diamond}$ be vertices on
the   medial    lattice   $\Omega_{\diamond}^{\delta}$    nearest   to
$x_*^{\delta}$  and   $y_*^{\delta}$.  Then,   we  define   the  Ising
\textit{interface} as follows. It starts from $x^{\delta}_{\diamond}$,
traverses on the primal lattice  $\Omega^{\delta}$, and turns at every
vertex of  $\Omega^{\delta}$ in  such a  way that  it always  has
dual-vertices with  spin $\oplus$  on its  left and  spin $\ominus$  on its
right. If there is an indetermination  when arriving at a vertex (this
may  happen  on   the  square  lattice),  it  turns   left.  See  also
Figure~\ref{fig::ising_multiple} for an illustration.
This interface converges weakly as $\delta \to 0$
to 
the chordal $\SLE_\kappa$ with $\kappa = 3$
(in the topology of Section~\ref{subsec::intro_Ising}).

\begin{figure} [h]
  \begin{center}
    \includegraphics[width=.6\textwidth]{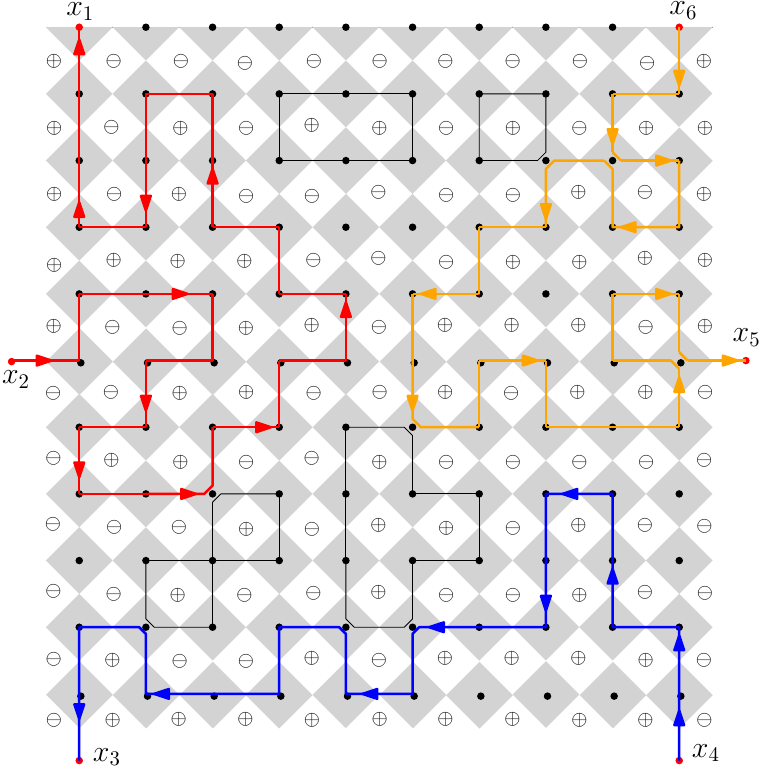}
  \end{center}
  \caption{A spin configuration of the Ising model on a polygon with six marked points $x_1, \ldots, x_6$ on the boundary,
    with alternating boundary conditions. There are three interfaces starting from $x_2$, $x_4$, and $x_6$, illustrated in red, blue, and orange, respectively.}
  \label{fig::ising_multiple}
\end{figure}

%

\begin{theorem}\label{thm::ising_cvg_minusplus}
  \textnormal{\cite[Theorem~1]{CDCHKSConvergenceIsingSLE}}
  Let $(\Omega_* ^{\delta}; x_*^{\delta}, y_*^{\delta})$ be a sequence
  of  discrete  Dobrushin domains  converging  to  a Dobrushin  domain
  $(\Omega;   x,   y)$  in   the   Carath\'eodory   sense.  Then,   as
  $\delta\to  0$,  the  interface  of  the  critical  Ising  model  on
  $(\Omega_*^{\delta},  x_*^{\delta},  y_*^{\delta})$  with  Dobrushin
  boundary conditions  converges weakly  to the chordal  $\SLE_{\kappa}$ in
  $\Omega$ connecting $x$ and $y$ with $\kappa=3$.
\end{theorem}

Using this result, we will prove that multiple interfaces also converge in the scaling limit to global multiple $\SLE_3$ curves.
Abusing and lightening notation, let us write $\Omega^{\delta}$ for $\Omega^{\delta}, (\Omega^{\diamond})^{\delta}$, or $(\Omega^*)^{\delta}$, and
$x^{\delta}$ for $x^{\delta}, (x^{\diamond})^{\delta}$, or $(x^*)^{\delta}$.
Let the polygons $(\Omega^{\delta}; x_1^{\delta}, \ldots, x_{2N}^{\delta})$
converge to $(\Omega; x_1, \ldots, x_{2N})$ as $\delta\to 0$ in the Carath\'eodory sense.
Consider the critical Ising model on $\Omega^{\delta}$ with alternating boundary conditions~\eqref{eqn::ising_bc_alternating}.
For $j\in\{1,\ldots, N\}$, let $\eta_j^{\delta}$ be the interface starting from $x_{2j}^{\delta}$ that separates $\oplus$ from $\ominus$. Then, the collection of interfaces
$(\eta_1^{\delta}, \ldots, \eta_N^{\delta})$ connects the boundary points $x_1^{\delta},\ldots, x_{2N}^{\delta}$ forming a planar link pattern $\conn^{\delta}\in \LP_N$.
Proposition~\ref{prop::ising_multiple_cvg} asserts that
conditionally on the event $\{\conn^{\delta}=\alpha\}$,
the law of the collection $(\eta^{\delta}_1, \ldots, \eta_N^{\delta})$ converges weakly as $\delta \to 0$ to a global $N$-$\SLE_{3}$ associated to $\alpha$.
The proof of this is very similar to that for the FK-Ising model (Proposition~\ref{prop::fkising_alternating_cvg})
--- we summarize it below.

\begin{proof}[Proof of Proposition~\ref{prop::ising_multiple_cvg}]
  The uniqueness of the limit follows from Theorem~\ref{thm::global_unique} (the global $N$-$\SLE_3$ is unique).
  For the subsequential convergence, we follow the same lines as in the proof of Proposition~\ref{prop::fkising_alternating_cvg}.
  Recall the summary of its proof from Remark~\ref{rem::summary}.
  First, for Lemmas~\ref{lem::tightness} and~\ref{lem::cvg_domains_Dobrushin}, we need a RSW type estimate for the critical Ising model.
  This can be obtained from Proposition~\ref{prop::rcm_rsw_rectangle} via the so-called Edwards-Sokal coupling,
  as explained in~\cite[Remark~4]{CDCHKSConvergenceIsingSLE}.
  Then, the proof of Lemma~\ref{lem::cvg_polygon_global_multiple_subtle} holds for the critical Ising model and $\kappa = 3$.
  Therefore, we conclude that for any convergent subsequence of $(\eta_1^{\delta_n}, \ldots, \eta_N^{\delta_n})_{\delta_n > 0}$,
  the limit must be a global multiple $N$-$\SLE_3$.
  Since the global $N$-$\SLE_3$ is unique due to Theorem~\ref{thm::global_unique}, we readily establish the convergence of
  the whole sequence to this global $N$-$\SLE_3$.
Finally,   the   asserted  marginal  law  of   $\eta_j$  follows  from  Lemma~\ref{lem::loewnerchain_purepartition}.
\end{proof}

\appendix

\section{Intersection of Two Fractals}
\label{sec::appendix_fractal}
For use in Section~\ref{sec::ising_fkperco}, we record in this appendix some properties of random subsets of the boundary of 
the unit disc $\U = \{ z \in \C \colon |z| \leq 1 \}$. 
In spite of stating the results for $\U$, we may as well
apply the following lemma for the domain $\Omega=[0,\ell]\times[0,1]$ as we do in Section~\ref{sec::ising_fkperco}, by conformal invariance of the $\SLE_\kappa$.

\begin{lemma}\label{lem::appendix}
Suppose $\LE$ is a random subset of $\partial\U$ satisfying the following: there are constants $C>0$ and $\beta\in (0,1)$ 
such that, for any interval $I \subset \partial\U$, we have
\begin{align*} 
\PP[\LE\cap I\neq\emptyset]\le C|I|^{\beta}.
\end{align*}
Then, for any subset $A\subset\partial\U$, the following hold.
\begin{enumerate}
\item \label{item::dim_empty} 
If $\dimH(A)<\beta$, then 
\begin{align*} 
A\cap\LE=\emptyset , \qquad \text{almost surely.}
\end{align*}
\item \label{item::dim_upper} 
If $\dimH(A)\ge\beta$, then 
\begin{align*} 
\dimH(A\cap\LE)\le \dimH(A)-\beta , \qquad \text{almost surely.} 
\end{align*}
\end{enumerate}
\end{lemma}

This lemma is a part of~\cite[Lemma~2.3]{RohdeWuFractals}, where the authors give a more complete description of the set $A\cap \LE$. 
The above cases are sufficient to our purposes in the proof of 
Proposition~\ref{prop::two_sle_beyond4}, 
so we include their proofs in this appendix.

\begin{proof}[Proof of item~\ref{item::dim_empty}]
Since $\beta>\dimH(A)$, for any $\varepsilon>0$, there exists a cover $\cup_i I_i$ of $A$ such that $\sum_i |I_i|^{\beta}\le\varepsilon$. Therefore,
\begin{align*}
\PP[A\cap\LE \neq \emptyset] \; \le \; \sum_i \PP[I_i\cap\LE\neq\emptyset] \; \le \; C\sum_i |I_i|^{\beta} \; \le \; C\varepsilon,
\end{align*}
almost surely. Letting $\varepsilon\to 0$, we see that $\PP[A\cap\LE\neq\emptyset]=0$.
\end{proof}

\begin{proof}[Proof of item~\ref{item::dim_upper}]
For any $\gamma>\dimH(A)-\beta$, there exists a cover $\cup_i I_i$ of $A$ such that $\sum_i|I_i|^{\beta+\gamma}<\infty$.  Hence, we have
\begin{align*}
\E \Big[\sum_i |I_i|^{\gamma} \one_{\{I_i\cap\LE\neq\emptyset\}} \Big] 
\, =  \, \sum_i |I_i|^{\gamma} \,\PP[I_i\cap\LE\neq\emptyset] \; \le \;  C\sum_i |I_i|^{\beta+\gamma}
\; < \; \infty,
\end{align*}
almost surely. Thus, the collection
$\{I_i \colon I_i\cap\LE\neq\emptyset \}$
is a cover of $A\cap\LE$ and $\sum_i |I_i|^{\gamma}1_{\{I_i\cap\LE\neq\emptyset\}}<\infty$, almost surely. Therefore, we have 
\[\dimH(A\cap\LE)\le \gamma, \qquad \text{almost surely.}\]
As this holds for any $\gamma>\dimH(A)-\beta$, we have $\dimH(A\cap\LE)\le \dimH(A)-\beta$, almost surely.
\end{proof}

\bigskip

{\small
\newcommand{\etalchar}[1]{$^{#1}$}

}

%
%
%
%
%

\end{document}